\documentclass[12pt]{article} 
\usepackage{graphicx}
\usepackage{amssymb}
\usepackage{mathtools}
\usepackage{amsmath}
\usepackage{amsthm}
\usepackage{color}
\usepackage[hidelinks]{hyperref}
\usepackage{csquotes}
\usepackage[pagewise]{lineno}
\usepackage[english]{babel}
\usepackage{enumitem}
\usepackage{graphicx,url} 
\usepackage{tocloft}
\usepackage{sectsty}
\usepackage{comment}

\renewcommand{\Re}{\operatorname{Re}}
\newtheorem{thm}{Theorem}[section] %

\newtheorem{lem}[thm]{Lemma}
\newtheorem{prop}[thm]{Proposition}
\newtheorem{defn}{Definition}[section] %
\newtheorem{exmp}{Example}[section] %
\newtheorem{rmk}{Remark}[section] %

\numberwithin{equation}{section}

\sectionfont{\fontsize{13}{15}\selectfont}
\subsectionfont{\fontsize{12}{15}\selectfont}

\def \R{{\mathbb{R}}}
\def \P {\Phi(x)}
\def \o {\omega(x)}
 
\def \japxi{\langle \xi \rangle}
\def \japxik{\langle \xi \rangle_k}
\def \japx{\langle x \rangle}
\def \hyp{Z_{ext}(N)}
\def \pd{Z_{int}(N)}

\def \J{[0,T] \times \R^n \times \R^n}
\def \la{\langle}
\def \ra{\rangle}
\def \ran{\rangle_k}
\def \rak{\rangle_k}

\def \wm{(\omega,g_{\Phi,k})}
\def \wmt{(\omega^{\tilde m_2}\Phi^{\varepsilon},g_{\Phi,k}^{(1,\delta),(1-\delta,0)})}
\def \wmtt{(\Phi^{\varepsilon'},g_{\Phi,k}^{(1,\delta),(1-\delta,0)})}

\providecommand{\keywords}[1]
{
	\small	
	\textbf{\text{Keywords:}} #1
}

\providecommand{\subclass}[1]
{
	\small	
	\textbf{\text{MSC(2010):}} #1
}

\usepackage{geometry}
\title{\Large Strictly Hyperbolic Cauchy Problems on $\R^n$ with  \\Unbounded and Singular Coefficients\thanks{Authors dedicate this work to Bhagawan Sri Sathya Sai Baba.}}

\author{\normalsize Rahul Raju Pattar\thanks{rahulrajupattar@sssihl.edu.in (Corresponding Author)} , N. Uday Kiran\thanks{nudaykiran@sssihl.edu.in}  \\
	\small Department of Mathematics and Computer Science\\
	\small Sri Sathya Sai Institute of Higher Learning, Puttaparthi, India \\
}

\date{}
\begin{document} 
	
	\maketitle
	
	\begin{abstract}
		We investigate the behavior of the solutions of a class of certain strictly hyperbolic equations defined on $(0,T]\times \R^n$ in relation to a class of metrics on the phase space. In particular, we study the global regularity and decay issues of the solution to an equation with coefficients polynomially bound in $x$ with their $x$-derivatives and $t$-derivative of order $\textnormal{O}(t^{-\delta}),\delta \in [0,1),$ and $\textnormal{O}(t^{-1})$ respectively. This type of singular behavior allows coefficients to be either oscillatory or logarithmically bounded at $t=0$. We use the Planck function associated with the metric to subdivide the extended phase space and define an appropriate generalized parameter dependent symbol class. We report that the solution not only experiences a finite loss of derivatives but also a decay in relation to the initial datum defined in a Sobolev space tailored to the metric. Our analysis suggests that an infinite loss is quite expected when the order of singularity of the first time derivative of the leading coefficients exceeds $O(t^{-1})$. We confirm this by providing counterexamples. Further, using the $L^1$ integrability of the logarithmic singularity in $t$ and the global properties of the operator with respect to $x$, we derive the anisotropic cone conditions in our setting.\\
		\\	
		\\
		\subclass{35L30 $\cdot$ 35S05 $\cdot$ 35B65 $\cdot$ 35B30}
	\end{abstract}
	
	
	\section{Introduction}
	A study of loss of derivatives in relation to the initial datum plays a prominent role in the modern theory of hyperbolic equations. It is now well established that the loss is caused due to either an irregularity, a singularity, or a degeneracy of the coefficients with respect to the time variable (see for instance \cite{CSK,Cico1,Cico2,hiro,GG,KuboReis,CicoLor,LorReis,Petkov}, and the references therein). Recently, many authors have also considered the interplay between the irregularities in time and suitable bounds on the space (see for instance \cite{AscaCappi1,AscaCappi2,NUKCori1,Rahul_NUK2}).   
	
	In this work, our main interest is on the optimality of loss when the coefficients are singular in time and unbounded in space. In particular, our interest is either infinitely many oscillations or logarithmic blow-up in $t$, and polynomially bounded on $x$. In order to study the influence of unboundedness in space one needs to consider an appropriate metric on the phase space \cite{Lern, nicRodi}. 
	
	The notion of a metric on phase space $T^{*}\R^{n}\ (\cong \R^{2n})$ was first introduced by H\"ormander \cite{Horm} who studied the smooth functions $p(x,\xi)$ called symbols, using the metric $\japxi^{2\delta} |dx|^2+ \japxi^{-2\rho} |d\xi|^2$, $0\leq \delta<\rho\leq 1$. That is, the symbol $p(x,\xi)$ satisfies, for some $m \in \R \text{ and } C_{\alpha\beta}>0$,
	\begin{linenomath*}
		\[
		|\partial_\xi^\alpha \partial_x^\beta p(x,\xi)| \leq C_{\alpha\beta} \japxi^{m-\rho|\alpha|+\delta|\beta|}, 
		\]
	\end{linenomath*}
	where $\alpha,\beta \in \mathbb{N}_0^n\;(\mathbb{N}_0=\mathbb{N} \cup \{0\})$ are multi-indices and $\japxi=(1+|\xi|^2)^{1/2}$. In a more general framework created by Beals and Fefferman \cite{Feff,nicRodi}, we can consider the symbol $p(x,\xi)$ that satisfies the estimate for some $C_{\alpha\beta}>0,$
	\begin{linenomath*}
		\begin{equation*}		
			|\partial_\xi^\alpha \partial_x^\beta p(x,\xi)| \leq C_{\alpha\beta} M(x,\xi) \Psi(x,\xi)^{-|\alpha|} \Phi(x,\xi)^{-|\beta|},
		\end{equation*}
	\end{linenomath*}
	where the positive functions $M(x,\xi), \Psi(x,\xi)$ and $\Phi(x,\xi)$ are specially chosen. In such a case, as given in \cite{Lern}, we can consider the following Riemannian structure on the phase space 
	\begin{linenomath*}
		\begin{equation}
			\label{beals_fefferman}
			g_{x,\xi}=\frac{|dx|^2}{\Phi(x,\xi)^2}+\frac{|d\xi|^2}{\Psi(x,\xi)^2};
		\end{equation}
	\end{linenomath*}
	thus, one has geometric restrictions on $\Phi(x,\xi)$ and $\Psi(x,\xi)$ of both Riemannian type and Symplectic type \cite{Lern}. 
	
	In this paper, we consider a metric (\ref{beals_fefferman}) with $\Phi(x,\xi)=\Phi(x)$ and $\Psi(x,\xi)=\japxik=(k^2+|\xi|^2)^{1/2}$, for a large positive parameter $k$, and the weight function $M(x,\xi)=\o^{m_2}\japxik^{m_1}$ for $m_1,m_2 \in \R.$ Here the functions $\o$ and $\P$ are positive monotone increasing in $|x|$ such that $1\leq \o \leq \P\lesssim \japx$ ($\japx := (1+\vert x\vert^2)^{1/2}$). Further, note that $\omega(x)$ and $\Phi(x)$ are associated with the weight and metric respectively, and they specify the structure of the differential equation in space variable. These functions will be discussed in detail in Section \ref{metric}.  
	
	Our interest is to study the behavior of solutions corresponding to a second order linear hyperbolic equation with coefficients that are polynomially growing $x$ and singular at $t=0$. Such equations have a well-known behavior of a loss of derivatives when the coefficients are bounded in space ($\R^n$) together with all their derivatives (see for example \cite{CSK,Cico1}). Along with the extension of the results to a global setting ($x \in \R^n$ and the coefficients are allowed to grow polynomially in $x$), we also investigate the behavior at infinity of the solution in relation to the coefficients. By a loss of regularity of the solution in relation to the initial datum we mean a change of indices in an appropriate Sobolev space associated with the metric (\ref{beals_fefferman}).
	 
	As a prototype of our model operator (see Section \ref{stmt}), consider the Cauchy problem:
		\begin{equation}
			\begin{rcases}
				\label{PO}
				\partial_t^2u - a(t,x)\Delta_xu = 0,  \quad (t,x) \in (0,T] \times \R^n, \\
				u(0,x)=u_1(x), \quad \partial_tu(0,x)=u_2(x),
			\end{rcases}
		\end{equation}
	where the coefficient $a(t,x)$ is in $C^1((0,T];C^\infty(\R^n))$ and satisfies the following estimates
	\begin{align}
	    \label{ell}
		a(t,x) & \geq C_{0} \o^2,\\ 
		\label{sing}
		|\partial_x^\beta\partial_t a(t,x)| & \leq C_{\beta} \frac{1}{t}  \left( \log \left(1+\frac{1}{t}\right) \right)^{|\beta|}\o^2\P^{-|\beta|}
	\end{align}
	where $\beta \in \mathbb{N}_0^n$, $C,C_{0},C_{\beta}>0$. From (\ref{sing}) with $|\beta| =0$, we have
	\begin{linenomath*}
	\[
	    \vert a(T,x) - a(t,x) \vert \leq \int_t^T \vert \partial_s a(s,x) ds \vert \leq C \log \left( \frac{T}{t} \right) \o^2.
	\]
	\end{linenomath*}
	Implying  
	\begin{equation}\label{logblow}
	    \vert a(t,x) \vert \leq C  \log \left( 1+\frac{1}{t} \right)\o^2.
	\end{equation}
	Thus, the coefficient $a(t,x)$ has a logarithmic blow-up at $t=0.$ 
	\begin{rmk}\label{rmk}
	   The singular behavior quantified by (\ref{sing}) is a generalization of Cicognani \cite{Cico1} to a global setting. Our conditions (\ref{sing}) confirm with the conditions on the coefficients given in \cite{Cico1} for $\o=\P=1$. Moreover, we can also replace (\ref{sing}) with a slightly general estimates, given in \cite{Cico1}, 
	\begin{equation}\label{sing2}
	    \begin{rcases}
	    \begin{aligned}
	        |\partial_x^\beta a(t,x)| & \leq C_{\beta} \frac{1}{t^{\delta_1}} \o^2\P^{-|\beta|}, \quad |\beta| >0,\\
			|\partial_x^\beta\partial_t a(t,x)| & \leq C_{\beta} \frac{1}{t^{1+\delta_2|\beta|}} \o^2\P^{-|\beta|}, \quad |\beta| \geq 0,
	    \end{aligned}
	    \end{rcases} \tag{\ref{sing}${}^*$}
	\end{equation}
	for $\delta_1,\delta_2 \in [0,1).$
	\end{rmk}
	An example of a coefficient $a(t,x)$ satisfying (\ref{ell}) and (\ref{sing2}) is given below.
	
	\begin{exmp}
		Let $n=1$, $T=1$, $\kappa_1 \in[0,1]$ and $\kappa_2\in(0,1]$ such that $\kappa_1 \leq \kappa_2$. Then, 
		\begin{linenomath*}
		$$
			a(t,x) = \japx^{2\kappa_1} \left(2+\sin \left( \japx^{1-\kappa_2} + \cos x \; \log t\right) + \left( 2+\cos  \japx^{1-\kappa_2} \right)\log(1+1/t)\right)
		$$
		\end{linenomath*}
		satisfies the estimates (\ref{sing2}) for $\o = \japx^{\kappa_1}$, $\P = \japx^{\kappa_2}$ and for any $\delta_1,\delta_2 \in (0,1)$.
	\end{exmp}
	As evident from (\ref{sing2}), we consider coefficients that are singular at $t=0$. Study of these equations is motivated by a behavioral study of blow-up solutions in quasilinear problems, see for example \cite[Section 4]{hiro}.
	In \cite{Cico1}, Cicognani discussed the well-posedness of (\ref{PO}) for the case $\o=\P=1$ in (\ref{ell}) and (\ref{sing2}). The author reports well-posedness in $C^\infty(\R^n)$ for 
	the Cauchy problem (\ref{PO}) with a finite loss of derivatives.
	
	In \cite{NUKCori1}, the second author along with Coriasco and Battisti extended the work of Kubo and Reissig \cite{KuboReis} to the SG setting, using the generalized parameter dependent Fourier integral calculus. The authors consider a weakly hyperbolic Cauchy problem with 
	\begin{linenomath*}
		\[
		\vert \partial_t^l \partial_x^\beta a(t,x)\vert \leq C_{l,\beta} \bigg(\frac{1}{t}\bigg(\ln \frac{1}{t}\bigg)^\gamma\bigg)^l \japx^{2-\vert \beta \vert}, 
		\]
	\end{linenomath*}
	where $(t,x) \in (0,T] \times \R^n, \; 0 \leq \gamma \leq 1,$ for all $l\in \mathbb{N}\textnormal{ and }\beta$ multi-index, for the model problem (\ref{PO}). In \cite{Rahul_NUK2}, we considered $m^{th}$ order equation and extended the results of Cicognani \cite{Cico1} to a global setting with respect to the metric (\ref{m1}). Here we assumed that the singularity on the first $t$-derivative of the leading coefficients is of order $q$, for $1\leq q <\frac{3}{2}$. Nevertheless, the coefficients in both the works \cite{NUKCori1,Rahul_NUK2} are continuous at $t=0$ and are bounded in time variable.
	
	We allow the coefficients to logarithmically blow-up near $t=0$ for the case $q=1$ as evident from the conditions (\ref{sing}) and (\ref{logblow}) and consider the Cauchy problem (\ref{PO}), where $a(t,x)$ satisfies the estimates in (\ref{ell}) and (\ref{sing2}). Further, we study the behavior at infinity of the solutions to the Cauchy problem by considering the metrics
	\begin{equation}\label{m1}
		g_{\Phi,k} = \P^{-2} \vert dx\vert^2 + \japxik^{-2}\vert d\xi\vert^2,
	\end{equation}  
	where $\japxik : = (k+\vert \xi\vert^2 )^{1/2}$ for a large parameter $k\geq1.$ Our methodology relies on using the Planck function (see Section \ref{metric} for the definition) associated to the metric $g_{\Phi,k}$ to subdivide the extended phase space (see Section \ref{zones}) and define appropriate generalized parameter dependent symbol classes. As one expects, we report that the solution not only experiences a loss of derivatives but also a decay in relation to the initial datum defined in a Sobolev space $H^{s}_{\Phi,k}(\R^n)$ tailored to the metric $g_{\Phi,k}$ (discussed in the Section \ref{metric}). 

	In Section \ref{cone}, we derive an optimal cone condition for the solution of the Cauchy problem (\ref{eq1}) from the energy estimate used in proving the well-posedness. Though the characteristics of the operator in (\ref{PO}) are singular in nature, the $L^1$ integrability of logarithmic singularity guarantees that the propagation speed is finite. The weight function governing the coefficients also influences the geometry of the slope of the cone in such a manner that the slope grows as $|x|$ grows.
	
    The methodology used in Section \ref{Proof1} suggests that an infinite loss in both decay and derivatives is quite expected phenomenon when the order of singularity of the first time derivative of the leading coefficients exceeds $O(t^{-1})$. In Section \ref{CE}, we justify this by providing. In fact, by extending the work of Ghisi and Gobbino \cite[Section 4]{GG} to a global setting, we show that the set of coefficients which result in infinite loss of regularity is nonempty and residual in certain metric space. The ideas here rely on  the spectral theorem for pseudodifferential operators on $\R^n$ and Baire category arguments.

	\section{Tools}
	
	In this section we introduce the main tools which we need for our analysis. The first part is devoted to a class of metrics on the phase space that govern the geometry of the symbols in our consideration. In the second part we device a localization technique based on the Planck function associated to the metric.
	
	\subsection{Our Choice of Metric on the Phase Space}\label{metric}
	In general, we use the metrics of the form
	\begin{equation}\label{m2}
		g_{\Phi,k}^{\rho,r} = \left( \frac{\japxik^{\rho_2}}{\P^{r_1}} \right)^2 |dx|^2 +  \left( \frac{\P^{r_2}}{\japxik^{\rho_1}} \right)^2 |d\xi|^2.
	\end{equation}
	Here $\rho=(\rho_1,\rho_2)$ , $r=(r_1,r_2)$ for $\rho_j,r_j \in [0,1]$, $j=1,2$ are such that $0 \leq \rho_2<\rho_1 \leq 1$ and $0 \leq r_2<r_1 \leq 1$. In this section we discuss the structure functions $\Phi$ and $\omega$. 
	
	Let us start by reviewing some notation and terminology used in the study of metrics on the phase space, see \cite[Chapter 2]{Lern} for details. Let us denote by $\sigma(X,Y)$ the standard symplectic form on $T^*\R^n\cong \R^{2n}$: if $X=(x,\xi)$ and $Y=(y,\eta)$, then
	\begin{linenomath*}
	\[
		\sigma(X,Y)=\xi \cdot y - \eta \cdot x.
	\]	
	\end{linenomath*}
	We can identify $\sigma$ with the isomorphism of $\R^{2n}$ to $\R^{2n}$ such that $\sigma^*=-\sigma$, with the formula $\sigma(X,Y)= \langle \sigma X,Y\rangle$. Consider a Riemannian metric $g_X$ on $\R^{2n}$ (which is a measurable function of $X$) to which we associate the dual metric $g_X^\sigma$ by
	\begin{linenomath*}
	\[ 
	\forall T \in \R^{2n}, \quad g_X^\sigma(T)= \sup_{0 \neq T' \in \R^{2n}} \frac{\langle \sigma T,T'\rangle^2}{g_X(T')}.
	\]
	\end{linenomath*}
	
	Considering $g_X$ as a matrix associated to positive definite quadratic form on $\R^{2n}$, $g_X^\sigma=\sigma^*g_X^{-1}\sigma$.
	We define the Planck function \cite{nicRodi}, that plays a crucial role in the development of pseudodifferential calculus to be
	\begin{linenomath*}
	\[ 
		h_g(x,\xi) := \inf_{0\neq T \in \R^{2n}} \Bigg(\frac{g_X(T)}{g_X^\sigma(T)}\Bigg)^{1/2}.
	\]
	\end{linenomath*}
	 The uncertainty principle is quantified as the upper bound $h_g(x,\xi)\leq 1$. In the following, we often make use of the strong uncertainty principle, that is, for some $\kappa>0$, we have
	 \begin{linenomath*}
	 \[
	 	h_g(x,\xi) \leq (1+|x|+|\xi|)^{-\kappa}, \quad (x,\xi)\in \R^{2n}.
	 \]
	 \end{linenomath*}
	 The Planck function associated to the metric in (\ref{m2}) is given by $\P^{r_2-r_1} \japxik^{\rho_2-\rho_1}.$
	 
	 Basically, a pseudodifferential calculus is the datum of the metric satisfying some local and global conditions. In our case, it amounts to the conditions on $\P$. The symplectic structure and the uncertainty principle also play a natural role in the constraints imposed on $\Phi$. So we consider $\P$ to be a monotone increasing function of $|x|$ satisfying following conditions: 
	 \begin{linenomath*}
	\begin{alignat}{3}
		\label{sl}
		1 \; \leq & \quad \Phi(x) &&\lesssim  1+|x| && \quad \text{(sub-linear)} \\
		\label{sv}
		\vert x-y \vert \; \leq & \quad r\Phi(y) && \implies C^{-1}\Phi(y)\leq \Phi(x) \leq C \Phi(y)  && \quad \text{(slowly varying)} \\
		\label{tp}
		&\Phi(x+y) && \lesssim  \Phi(x)(1+|y|)^s && \quad \text{(temperate)}
	\end{alignat}
	\end{linenomath*}
	for all $x,y\in\R^n$ and for some $r,s,C>0$. Note that $C \geq 1$ in the slowly varying condition with $x=y$. 
	
	For the sake of calculations arising in the development of symbol calculus related to metrics $g_\Phi$ and $\tilde{g}_\Phi$, we need to impose following additional conditions:
	\begin{linenomath*}
	\begin{alignat}{3}
		\label{sa}
		|\Phi(x) - \Phi(y)| \leq & \Phi(x+y) && \leq \Phi(x) + \Phi(y),  && \quad  (\text{Subadditive})\\
		\label{Phi}
		& |\partial_x^\beta \Phi(x)| && \lesssim \Phi(x) \japx^{-|\beta|}, \\
		&\Phi(ax) &&\leq a\P, \text{ if } a>1,\\
		\label{scale}
		& a\P &&\leq \Phi(ax), \text{ if } a \in [0,1],
	\end{alignat}
	\end{linenomath*}
	where $\beta \in \mathbb{Z}_+^n$. It can be observed that the above conditions are quite natural in the context of symbol classes. In our work, we need even the weight function $\o$ to satisfy the properties (\ref{sl})-(\ref{scale}). In arriving at the energy estimate using the Sharp G\r{a}rding inequality (see Section \ref{energy} for details), we need 
	\begin{linenomath*}
	\[
		\o \leq \P, \quad x \in \R^n.
	\] 
	\end{linenomath*}
	The Sobolev spaces related to the metric $g_{\Phi,k}$ in (\ref{m1}) are defined below.
	\begin{defn} \label{Sobo}   
		The Sobolev space $H^{s}_{\Phi,k}(\R^n)$ for $s=(s_1,s_2) \in \R^2$ and $ k \geq 1,$ is defined as
			\begin{equation}
				\label{Sobo2}
				H^{s}_{\Phi,k}(\R^n) = \{v \in L^2(\R^n): \P^{s_2}\langle D \ran^{s_1}v \in L^2(\R^{n}) \},
			\end{equation}
		equipped with the norm
		$
		\Vert v \Vert_{\Phi,k;s} = \Vert \Phi(\cdot)^{s_2}\langle D \ran^{s_1}v \Vert_{L^2} .
		$ 
	\end{defn}
    The subscript $k$ in the notation $H^{s}_{\Phi,k}(\R^n)$ is related to the parameter in the operator $\la D \ra_k = (k^2 - \Delta_x)^{1/2}.$ When $k=1$, we denote the space $H^{s}_{\Phi,1}(\R^n)$ simply as $H^{s}_{\Phi}(\R^n).$
    
	\subsection{Subdivision of the Phase Space}\label{zones}
	One of the main tools in our analysis is the division of the extended phase space into two regions using the Planck function, $h(x,\xi)=(\P \japxik)^{-1}$ of the metric $g_{\Phi,k}$ in (\ref{m1}). We use these regions in the proof of Theorem \ref{result1} (see Section \ref{factr}) to handle the low regularity in $t$. To this end we define $t_{x,\xi}$, for a fixed $(x,\xi)$, as the solution to the equation
	\begin{linenomath*}
		\[
		t=\frac{N}{\P\japxik},
		\]
	\end{linenomath*}
	where $N$ is the positive constant. Using $t_{x,\xi}$ and the notation $J=\J$ we define the interior region
	\begin{linenomath*}
		\begin{equation} \label{zone1}
			\pd =\{(t,x,\xi)\in J : 0 \leq t \leq t_{x,\xi}\}
		\end{equation}
	\end{linenomath*}
	and the exterior region
	\begin{linenomath*}
		\begin{equation} \label{zone2}
			\hyp =\{(t,x,\xi)\in J : t_{x,\xi} < t \leq T\}.
		\end{equation}
	\end{linenomath*}	
	We use these regions to define the parameter dependent global symbol classes in Section \ref{Symbol classes}.

	\section{Statement of the Result}\label{stmt}
	Before we state the main result of this paper, we introduce our model strictly hyperbolic equation. We refer to Section \ref{Proof1} for the proof. 
	
	We deal with the Cauchy problem 
	\begin{linenomath*}
		\begin{equation}
		\begin{cases}
		\label{eq1}
		P(t,x,\partial_t,D_x)u(t,x)= f(t,x), \qquad D_x = -i\nabla_x,\;(t,x) \in (0,T] \times \R^n, \\
		u(0,x)=f_1(x), \quad \partial_tu(0,x)=f_2(x),
		\end{cases}
		\end{equation}
	\end{linenomath*}
	with the strictly hyperbolic operator $P(t,x,\partial_{t},D_{x}) = \partial_t^2 + a(t,x,D_x)+ b(t,x,D_x)$ where
	\begin{linenomath*}			
	\[
		a(t,x,\xi)  = \sum_{i,j=1}^{n} a_{i,j}(t,x)\xi_i\xi_j \quad \text{ and } \quad
		b(t,x,\xi)  = i\sum_{j=1}^{n} b_{j}(t,x)\xi_j + b_{n+1}(t,x).
	\]
	\end{linenomath*}
	Here, the matrix $(a_{i,j}(t,x))$ is real symmetric for all $(t,x)\in (0,T] \times \R^n$, $a_{i,j} \in C^1((0,T];C^\infty(\R^n))$ and $b_j \in C([0,T];C^\infty(\R^n))$. Similar to the estimates in Remark \ref{rmk}, we have the following assumptions on the operator $P$
	\begin{linenomath*}
			\begin{align}
			\label{elli}
			a(t,x,\xi) &\geq C_0 \o^2 \japxik^2, \quad C_0>0, \\
			\label{bound}
			\vert \partial_\xi^\alpha \partial_x^\beta a(t,x,\xi) \vert &\leq C_{\alpha\beta} \frac{1}{t^{\delta_1}} \o^2 \P^{-\vert \beta \vert}\japxik^{2-\vert \alpha \vert},\quad |\alpha| \geq 0, |\beta|>0, \\
			\label{B-up2}
			\vert \partial_\xi^\alpha \partial_x^\beta \partial_t a(t,x,\xi) \vert & \leq C_{\alpha\beta} \frac{1}{t^{1+\delta_2|\beta|}} \o^2 \P^{-\vert \beta \vert} \japxik^{2-\vert \alpha \vert},\quad |\alpha| \geq 0, |\beta|\geq 0,\\
			\label{Lower}
			\vert \partial_\xi^\alpha \partial_x^\beta b(t,x,\xi) \vert &\leq C_{\alpha\beta} \o \P^{-\vert \beta \vert}\japxik^{1-\vert \alpha \vert},
			\end{align}
	\end{linenomath*}		
	$\delta_1,\delta_2 \in \big[0,1\big), (t,x,\xi) \in [0,T] \times \R^n\times \R^n$. Note that $C_{\alpha\beta}(>0)$ is a generic constant which may vary from equation to equation.

	We now state the main result of this paper.
	Let $e=(1,1).$
	
	\begin{thm}\label{result1}
		Consider the strictly hyperbolic Cauchy problem  (\ref{eq1}) satisfying the conditions (\ref{elli}) - (\ref{Lower}). Let the initial data $f_j$ belong to $H^{s+(2-j)e}_{\Phi,k}$, $j=1,2$ and the right hand side $f \in C([0,T];H^{s}_{\Phi,k})$.
		Then, denoting $\delta=\max\{\delta_1,\delta_2\}$, for every $\varepsilon \in (0,1-\delta)$ there are $\kappa_0,\kappa_1 >0$ such that for every $s \in \R^2$ there exists a unique global solution
		\begin{linenomath*}
			\[
			u \in C\left([0,T];H^{s+(1-\Lambda(t))e}_{\Phi,k}\right)\bigcap C^{1}\left([0,T];H^{s-\Lambda(t)e}_{\Phi,k}\right),
			\]
		\end{linenomath*}
		where $\Lambda(t)=\kappa_0 + \kappa_1 t^\varepsilon/\varepsilon$. More specifically, the solution satisfies the a-priori estimate		
		\begin{linenomath*}
			\begin{equation}
			\begin{aligned}
			\label{est2}
			\sum_{j=0}^{1} \Vert \partial_t^ju(t,\cdot) \Vert_{\Phi,k;s+(1-j-\Lambda(t))e} \; &\leq C \Bigg(\sum_{j=1}^{2} \Vert f_j\Vert_{\Phi,k;s+(2-j)e} 
			+ \int_{0}^{t}\Vert f(\tau,\cdot)\Vert_{\Phi,k;s-\Lambda(\tau)e}\;d\tau\Bigg)
			\end{aligned}
			\end{equation}
		\end{linenomath*}
		for $0 \leq t \leq T, \; C=C_s>0$.				
	\end{thm}
	
	\section{Parameter Dependent Global Symbol Classes}	\label{Symbol classes}
	We now define certain parameter dependent global symbols that are associated with the study of the Cauchy problem (\ref{eq1}) which
	generalize the symbols given in \cite{Cico1} to a global setting. Let $m=(m_1,m_2)\in \mathbb{R}^2$. Consider the metric $g_{\Phi,k}^{\rho,r}$ as in (\ref{m2}).
	
	\begin{defn}
		$G^{m_1,m_2}(\omega,g_{\Phi,k}^{\rho,r})$ is the space of all functions $p=p(x,\xi) \in C^\infty(\mathbb{R}^{2n})$ satisfying 
		\begin{linenomath*}
		\begin{equation*}
		\label{sym1}
		\sup_{\alpha,\beta \in \mathbb{N}^n} \sup_{(x,\xi)\in \R^n}  \japxik^{-m_1+\rho_1|\alpha|-\rho_2|\beta|} \P^{-m_2+r_1|\beta|-r_2|\alpha|} |\partial_\xi^\alpha  D_x^\beta p(x,\xi)| < +\infty
		\end{equation*}	
		\end{linenomath*}		
	\end{defn}
	For the sake of simplicity, we denote the metric $g^{(1,0),(1,0)}_{\Phi,k}$ and the corresponding symbol class $G^{m_1,m_2}(\omega,g^{(1,0),(1,0)}_{\Phi,k})$ as $g_{\Phi,k}$ and $G^{m_1,m_2}(\omega,g_{\Phi,k})$, respectively.
	
	Observe that the derivatives of $\sqrt{a(t,x,\xi)}$ show enhanced singular behaviour compared to $a(t,x,\xi)$. Thus, to handle the singular behavior of the characteristics, we have the following symbol classes.
	\begin{defn}
		$G^{m_1,m_2}\{l_1,l_2;p\}_{int,N}(\omega,g_{\Phi,k})$ for $l_1,l_2\in \R$ and $p\in[0,1)$is the space of all $t$-dependent symbols $a=a(t,x,\xi)$ in $C^1((0,T];G^{m_1,m_2}(\omega,g_{\Phi,k}))$ satisfying 
		\begin{linenomath*}
		\begin{align*}
			\sup_{(t,x,\xi)\in \pd} \vert  \partial_\xi^\alpha a(t,x,\xi) \vert & \leq C_{00} \japxik^{m_1-|\alpha|} \o^{m_2} \bigg(\log\bigg(1+\frac{1}{t}\bigg)\bigg)^{l_1},\\
			\sup_{(t,x,\xi)\in \pd} \vert \partial_\xi^\alpha D_x^\beta a(t,x,\xi) \vert  &\leq C_{\alpha \beta} \japxik^{m_1-|\alpha|} \o^{m_2}\P^{-|\beta|} \left(\frac{1}{t}\right) ^{pl_2}
		\end{align*}
		\end{linenomath*}
		for some $C_{\alpha \beta}>0$ where $\alpha\in \mathbb{N}^n_0$ and $\beta \in \mathbb{N}^n.$ 
	\end{defn}	

	\begin{defn}
		$G^{m_1,m_2}\{l_1,l_2,l_3,l_4;p\}_{ext,N}(\omega,g_{\Phi,k})$ for $l_j\in \R,j=1,\dots,4$ and $p\in[0,1)$ is the space of all $t$-dependent symbols $a=a(t,x,\xi)$ in $C^1((0,T];G^{m_1,m_2}(\omega,g_{\Phi,k}))$ satisfying 
		\begin{linenomath*}
		\begin{align*}
			\sup_{(t,x,\xi)\in \hyp} \vert \partial_\xi^\alpha D_x^\beta a(t,x,\xi) \vert & \leq C_{\alpha \beta} \japxik^{m_1-|\alpha|} \o^{m_2}\P^{-|\beta|} \bigg(\frac{1}{t}\bigg)^{l_1+p(l_2+|\beta|)}  \\ & \qquad \times \bigg(\log\bigg(1+\frac{1}{t}\bigg)\bigg)^{l_3+l_4(|\alpha|+|\beta|)} 
		\end{align*}
		\end{linenomath*}
		for some $C_{\alpha \beta}>0$ where $\alpha,\beta \in \mathbb{N}^n_0.$ 
	\end{defn}	
	Given a $t$-dependent global symbol $a(t,x,\xi)$, we can associate a pseudodifferential operator $Op(a)=a(t,x,D_x)$ to $a(t,x,\xi)$ by the following oscillatory integral
	\begin{linenomath*}
	\begin{align*}
	a(t,x,D_x)u(t,x)& =\iint\limits_{\mathbb{R}^{2n}}e^{i(x-y)\cdot\xi}a(t,x,\xi){u}(t,y)dyd \xi\\
	& = (2\pi)^{-n}\int\limits_{\mathbb{R}^n}e^{ix\cdot\xi}a(t,x,\xi)\hat{u}(t,\xi) d\xi.
	\end{align*}
	\end{linenomath*}
	As for the calculus of symbol classes $G^{m_1,m_2}(\omega,g_{\Phi,k}^{\rho,r})$, we refer to \cite[Section 1.2 \& 3.1]{nicRodi} and \cite{CT}. The calculus for the operators with symbols in the additive form $a(t,x,\xi) = a_1(t,x,\xi) + a_2(t,x,\xi)$, for
	\begin{linenomath*}
	\[
		\begin{aligned}
			a_1 &\in G^{\tilde m_1,\tilde m_2}\{\tilde l_1,\tilde l_2;\delta_1\}_{int,N} \wm, \\
			a_2 &\in G^{m_1,m_2}\{l_1,l_2,l_3,l_4;\delta_2\}_{ext,N}\wm,
		\end{aligned}
	\]
	\end{linenomath*}
	is given in Appendix.
   \section{Proof of Theorem \ref{result1}} \label{Proof1}
	In this section, we give a proof of the main result. There are three key steps in the proof of Theorem \ref{result1}. First, we modify the coefficients of the principal part by performing an excision so that the resulting coefficients are regular at $t=0$. Second, we reduce the original Cauchy problem to a Cauchy problem for a first order system (with respect to $\partial_t$). Lastly, using sharp G\r{a}rding’s inequality we arrive at the $L^2$-well-posedness of a related auxiliary Cauchy problem, which gives well-posedness of the original problem in the  weighted Sobolev spaces $H^{s_1,s_2}_{\Phi,k}$.

	\subsection{Factorization}\label{factr}
	From (\ref{B-up2}), we observe that $a$ is $L^1$ integrable in $t \in [0,T]$. More precisely, $a(t,x,\xi)$ is logarithmically bounded at $t=0$, i.e.,
 	\begin{linenomath*}
 		\begin{equation}
 			\label{log}
 			|a(t,x,\xi)| \leq C  \o^{2} \japxik^2 \log (1+{1}/{t}), \quad C>0.
 		\end{equation}
 	\end{linenomath*}
 	We modify the symbol $a$ in $Z_{int}(2)$, by defining
	\begin{linenomath*}
	\begin{equation}\label{exci}
	\tilde{a}(t,x,\xi)=\varphi(t\P \japxik)\o^{2} \japxik^{2} + (1-\varphi(t\Phi(x) \japxik))a(t,x,\xi)
	\end{equation}
	\end{linenomath*} 
	for $
	\varphi \in C^\infty(\mathbb{R}) \text{ , }0\leq \varphi \leq 1 \text{ , } \varphi=1 \text{ in }[0,1] \text{ , }\varphi=0 \text{ in }[2,+\infty).$
	Note that $(a-\tilde{a}) \in G^{2,2}\{1,1;\delta_1\}_{int,2}(\omega,g_{\Phi,k})$ and $(a-\tilde{a}) \sim 0$ in $Z_{ext}(2).$ This
	implies that $t^{\delta_1}(a-\tilde{a})$ for $ t \in [0,T]$ is a bounded and continuous family in $G^{2,2}(\omega,g_{\Phi,k})$. 
	Observe that $a-\tilde{a}$ is $L^1$ integrable in $t$, i.e.,
	\begin{equation}\label{diff}
	\begin{aligned}
		\int^{T}_{0}\vert (a-\tilde{a})(t,x,\xi)\vert dt 
		&\leq \kappa_{0}  \o^{2} \japxik^{2} \int_{0}^{2/\Phi(x) 	\japxik}\log(1+1/t)dt	\\
		&\leq \kappa_{0} \o \japxi \log(1+\Phi(x) \japxik).
	\end{aligned}
	\end{equation}
	For a region $S$ in the extended phase space $\J$,  let $\chi$ be an indicator function defined by
	\[
		\chi(t,x,\xi) = \begin{cases}
			 1 & \text{ if } (t,x,\xi)  \in S\\
			 0 & \text{ otherwise}.
			\end{cases}
	\]
	In short we denote the indicator function for the region $S$ by $\chi(S)$. 
	
	Let $\tau(t,x,\xi)= \sqrt{\tilde{a}(t,x,\xi)}$. It is easy (by induction on $|\alpha|+|\beta|$) to note that 
	\begin{enumerate}[label=\roman*)]
		\item $\tau(t,x,\xi)$ is $G_\Phi$-elliptic symbol of order $(1,1)$ i.e. there is $C>0$ such that all $(t,x,\xi)\in [0,T] \times \mathbb{R}^n \times \mathbb{R}^n$ we have
		\begin{linenomath*}
		\[
		\vert\tau(t,x,\xi)\vert \geq C  \o \japxik.
		\]
		\end{linenomath*}
		\item $\tau(t,x,\xi)$ belongs to $G^{1,1}\{0,0;0\}_{int,2}(\omega,g_{\Phi,k}) + G^{1,1}\{0,0,1,1;\delta_1\}_{ext,1}(\omega,g_{\Phi,k}) $.
		\item Let $\delta=\max\{\delta_1,\delta_2\}$.  $\partial_t \tau \sim 0$ in $Z_{int}(1)$ and $\partial_t \tau \in G^{1,1}\{1,0,1,1;\delta\}_{ext,1}(\omega,g_{\Phi,k})$. To be precise,  there are $C_0,C_{\alpha \beta}>0$ such that for $(t,x,\xi)\in [0,T] \times \mathbb{R}^n \times \mathbb{R}^n$ and $|\alpha|\geq0,|\beta|>0$ we have
		\begin{linenomath*}
		\begin{equation*}
			\begin{aligned}
				{|\partial_t\tau(t,x,\xi)|}  & \leq C_{0} \; \chi(Z_{ext}(1))\japxik \o  (1/t) \log(1+1/t),\\
				{|\partial_\xi^\alpha D_x^\beta \partial_t \tau(t,x,\xi)|} & \leq C_{\alpha \beta} \; \chi(Z_{ext}(1))\japxik^{1-|\alpha|} \o \P^{-|\beta|}   \frac{1}{t^{1+\delta|\beta|}}\bigg(\log\bigg(1+\frac{1}{t}\bigg) \bigg)^{|\alpha|+|\beta|}.
			\end{aligned}
		\end{equation*}
		\end{linenomath*}
	\end{enumerate}  
	Consider $\varepsilon,\varepsilon'$ such that $0<\varepsilon< \varepsilon'<1-\delta$. Then by the definitions of the interior and exterior regions, we have
	\begin{enumerate}[label=\roman*)]\addtocounter{enumi}{3}
			\item $\tau \in C([0,T]; G^{1+\varepsilon,1}(\omega\Phi^\varepsilon,g^{(1-\varepsilon,\delta_1),(1-\delta_1-\varepsilon,0)}_{\Phi,k})$. 
			\item $\tau^{-1} \in C([0,T]; G^{-1,-1}(\omega,g^{(1-\varepsilon,\delta_1),(1-\delta_1-\varepsilon,0)}_{\Phi,k})$.
			\item $t^{1-\varepsilon}\partial_t \tau(t,\cdot,\cdot) \in G^{1+\varepsilon',1}(\omega\Phi^{\varepsilon'},g_{\Phi,k}^{(1,\delta),(1-\delta,0)})$, for all $t \in [0,T]$.
			\item $t^{1-\varepsilon}(\tilde{a}(t,x,D_x) -\tau(t,x,D_x) ^2) \in C\Big([0,T];OPG^{1,1}(\omega,g_{\Phi,k}^{(1,\delta_1),(1-\delta_1,0)})\Big) $. 
			\item $t^{1-\varepsilon} (a-\tilde a)(t,x,D_x) \in OPG^{1,1}(\omega,g_{\phi,k}).$
	\end{enumerate}

	We are interested in the factorization of the operator $P(t,x,\partial_t,D_x)$. This leads to
	\begin{linenomath*}
	\begin{equation*}
		P(t,x,\partial_t,D_x) = (\partial_t-i\tau(t,x,D_x)) 	(\partial_t+i\tau(t,x,D_x))+ (a-\tilde{a})(t,x,D_x) + a_1(t,x,D_x)
	\end{equation*}
	\end{linenomath*}
	where the operator $a_1(t,x,D_x)$ is such that, for $t \in [0,T]$, 
	\begin{linenomath*}
	\begin{equation*}
		a_1 = -i[\partial_t,\tau] + \tilde{a} -\tau^2 + b \text{ and } t^{1-\varepsilon}a_1(t,x,D_x) \in OPG^{1+\varepsilon',1}(\omega\Phi^{\varepsilon'},g_{\Phi,k}^{(1,\delta),(1-\delta,0)}).
	\end{equation*}
	\end{linenomath*}We define the functions $\psi_0(t,x,\xi)$ and $\psi_1(t,x,\xi)$ in $L^1([0,T];C^\infty(\R^{2n}))$ as
	\begin{equation}
		\label{major}
	\begin{aligned}
		\psi_0(t,x,\xi) &= C_1\varphi(t\P\japxik)\log(1+1/t)\o\japxik, \\
		\psi_1(t,x,\xi) &=  C_2\left(\varphi(t\P\japxik)\log(1+1/t)\o\japxik + (1-\varphi(t\P\japxik))\frac{1}{t} \right), 
	\end{aligned}
	\end{equation}
	for some $C_1,C_2>0,$ such that
	\begin{linenomath*}
	\begin{equation*}
	\frac{|(a-\tilde{a})(t,x,\xi)|}{\o\japxik}  \leq \psi_0(t,x,\xi) \;\text{  and } \; \frac{a_1(t,x,\xi)}{\o\japxik} \leq \psi_1(t,x,\xi).
	\end{equation*}	
	\end{linenomath*}
	Let $\psi = \psi_0 + \psi_1$. We observe that $t^{1-\varepsilon} \psi \in C\left([0,T]; G^{\varepsilon',\varepsilon'}(\Phi,g_{\phi,k}) \right)$ 
	\begin{equation}\label{mfn}
	\begin{aligned}
		\int^{T}_{0}|\psi (t,x,\xi)|dt 
		&\leq C\left( \o \japxik \int_{0}^{2/\Phi(x) \japxik}\log(1+1/t)dt+\int_{1/\Phi(x) \japxik}^{T}\frac{1}{t}dt \right) \\
		&\leq C \log(1+\Phi(x) \japxik),
	\end{aligned}
	\end{equation}
	while for $|\alpha| + |\beta|>0$ we have
	\begin{linenomath*}
	\begin{equation*}
		\int^{T}_{0}|\partial_{\xi}^{\alpha}D_x^{\beta}\psi (t,x,\xi)|dt 
		\leq C \P^{-|\beta|} \japxik^{-|\alpha|}\log(1+\Phi(x) \japxik) \chi(Z_{int}(2)).
	\end{equation*}
	\end{linenomath*}
	The function $\psi$ is used for making a change of variable while arriving at the energy estimate.	
	
	\begin{rmk}
		A general philosophy in dealing with hyperbolic equations with singular coefficients in $t$ is to subdivide the phase space into regions and perform a `surgery' on the irregularity of the symbol (as in equation (\ref{exci})). In our case, as the symbol is $L^1$ integrable in the region containing the singularity (\ref{diff}), we can majorize the coefficients by a logarithmic weight as in (\ref{mfn}).
	\end{rmk}

\subsection{First Order Pseudodifferential System}
We will now reduce the operator $P$ to an equivalent first order $2\times2$ pseudodifferential system. The procedure is similar to the one used in \cite{Cico1}. To achieve this, we introduce the change of variables $U=U(t,x)=(u_1(t,x),u_{2}(t,x))^T$, where
\begin{equation}
\label{COV}
\begin{cases}
u_1(t,x)= (\partial_t+i\tau(t,x,D_x))u(t,x), \\ 
u_2(t,x)=  \o \la D_x \rak u(t,x) - H(t,x,D_x)u_1, \\  
\end{cases}
\end{equation}
and the operator $H$ with the symbol $\sigma(H)(t,x,\xi)$ is such that
\begin{linenomath*}
\[
	\sigma(H)(t,x,\xi) = -\frac{i}{2}\o \japxik  \frac{\Big(1-\varphi\Big(t\P \japxik/3\Big)\Big)}{\tau(t,x,\xi)}.
\]
\end{linenomath*}
 Note that by the definition of $H$,  $\text{supp } \sigma(H) \cap \text{supp } \sigma(a-\tilde{a}) = \emptyset$ and we have 
 \begin{linenomath*}
\begin{align*}
	\sigma(2iH(t,x,D_x) \circ \tau(t,x,D_x) )&\sim 0,  \quad \text{ in } Z_{int}(3),\\
	\sigma(2iH(t,x,D_x) \circ \tau(t,x,D_x) )&= \o \la \xi \ra_k (1+\sigma(K_1)), \quad \text{ in } Z_{ext}(3),
\end{align*}
\end{linenomath*}
where $K_1 \in OPG^{-1,-1}_{int,3}\{0,0;\delta_1\}(\omega,g_{\Phi,k}) + OPG^{-1,-1}_{ext,3}\{0,1,2,3;\delta_1\}(\omega,g_{\Phi,k})$. Then, the equation $Pu=f$ is equivalent to the first order $2\times2$ system  :
\begin{equation}
	\label{FOS1}
	\begin{aligned}
		LU &= (\partial_t + \mathcal{D}+A_0+A_1)U=F,\\
		U(0,x)&=(f_2+i\tau(0,x,D_x)f_1,\P\la D_x\ra f_1)^T ,
	\end{aligned}
\end{equation}
	where
	\begin{linenomath*}
\begin{align*}
	F&=(f(t,x) ,-H(t,x,D_x)f(t,x) )^T,\\
	\mathcal{D} &= \text{diag}(-i\tau(t,x,D_x),i\tau(t,x,D_x)),\\
	A_0 &= \begin{pmatrix}
		B_0H & B_0 \\
		-HB_0H & HB_0
	\end{pmatrix}
	= \begin{pmatrix}
		\mathcal{R}_1 & B_0 \\
		-\mathcal{R}_3 & \mathcal{R}_2
	\end{pmatrix},\\
	A_1 &= \begin{pmatrix}
		B_1H & B_1 \\
		2iH\tau-M+i[M ,\tau]M^{-1}H + i[\tau, H]-HB_1H+\partial_tH & \;\; i[M,\tau]M^{-1}-HB_1
	\end{pmatrix}.
\end{align*}
\end{linenomath*}
The operators $M, M^{-1}, \tilde{M},B_0$ and $B_1$ are as follows
\begin{linenomath*}
	\begin{align*}
		M &= \o\la D_x \rak, \\
		M^{-1} &= \la D_x \rak^{-1}  \o^{-1},  \\
		B_0 &= (a(t,x,D_x) -\tilde a (t,x,D_x))\la D_x \rak^{-1}  \o^{-1}, \; \text{ and } \\
		B_1 &= (-i\partial_t\tau(t,x,D_x) + \tilde{a}(t,x,D_x) -\lambda(t,x,D_x)^2 + b(t,x,D_x)) \la D_x \rak^{-1}  \o^{-1}.
	\end{align*}
	\end{linenomath*}
	By the definition of operator $H$, we have $B_0H = \mathcal{R}_1, HB_0 =\mathcal{R}_2$, $HB_0H=\mathcal{R}_3$ for $\mathcal{R}_j\in G^{-\infty,-\infty}(\omega, g_{\Phi,k}),j=1,2,3,$ and the operator $2iH\tau-M$ is such that
	\begin{linenomath*}
	\[
		\sigma(2iH\tau-M) = \begin{cases}
			-\o\la \xi \rak, & \text{ in } Z_{int}(3),\\
			\o\la \xi \rak \sigma(K_1), & \text{ in } Z_{ext}(3).
		\end{cases}
	\]
	\end{linenomath*}
	The symbols of operators $A_0$ and $A_1$ are in the following symbol classes
	\begin{linenomath*}
	\begin{align*}
		\sigma(A_0) &\in G^{1,1}\{1,1;\delta_1\}_{int,3}(\omega,{g}_{\Phi,k}) + G^{-\infty,-\infty}\{0,0,0,0;0\}_{ext,3}(\omega,{g}_{\Phi,k}), \\
		\sigma(A_1) & \in G^{1,1}\{1,1;\delta_1\}_{int,1}(\omega,{g}_{\Phi,k}) + G^{0,0}\{1,0,1,1;\delta\}_{ext,1}(\omega,{g}_{\Phi,k})\\ & \qquad+G^{0,0}\{0,1,2,3;\delta\}_{ext,1}(\omega,{g}_{\Phi,k})
	\end{align*}
	\end{linenomath*}
	and thus, 
	\begin{linenomath*}
	\begin{align*}
		t^{1-\varepsilon}\sigma(A_0(t)) &\in C\left([0,T]; G^{1,1}(\omega,{g}_{\Phi,k}) \right),\\
		t^{1-\varepsilon}\sigma(A_1(t)) &\in C\left([0,T];G^{\varepsilon',\varepsilon'}(\omega,{g}_{\Phi,k}^{(1,\delta),(1-\delta,0)}) \right).
	\end{align*}
	\end{linenomath*}
	As in (\ref{major}), one can define positive functions $\tilde \psi_0,\tilde \psi _1\in L^1([0,T];C^\infty(\R^n)) \cap C^1((0,T];C^\infty(\R^n))$ 
	\begin{equation}\label{tpsi}
			\begin{aligned}
	 		\tilde \psi_0(t,x,\xi) &=  C_0\varphi(t\P\japxik/3)\log(1+1/t)\o\japxik \\
	 		\tilde \psi_1(t,x,\xi) &=  C_1\left(\varphi(t\P\japxik/3)\log(1+1/t)\o\japxik + (1-\varphi(t\P\japxik))\frac{1}{t} \right)
	 		\end{aligned}
	\end{equation}
	for some $C_0,C_1>0$, satisfying the estimates
	\begin{equation}
		\label{bd}
		\begin{aligned}
				|\sigma(A_0)| + |\sigma(A_1)| & \leq \tilde \psi, \\
				\int^{T}_{0}|\tilde \psi (t,x,\xi)|dt
					&\leq \kappa_{00} \log(1+\Phi(x) \japxik)\quad \text{ and } \\
					\int^{T}_{0}|\partial_\xi^\alpha D_x^\beta \tilde \psi (t,x,\xi)|dt
					&\leq \kappa_{\alpha \beta}\Phi(x)^{-|\beta|} \japxik^{-|\alpha|}\log(1+\Phi(x) \japxik) \chi(Z_{int}(6)),
		\end{aligned}
	\end{equation}
	where $\tilde \psi = \tilde \psi_0+\tilde \psi_1$ and $|\alpha| + |\beta| >0.$
\subsection{Energy estimate} \label{energy}
In this section, we prove the estimate (\ref{est2}). Note that it is sufficient to consider the case $s=(0,0)$ as the operator $\P^{s_2} \la D\ra^{s_1}L\la D\ra^{-s_1}\P^{-s_2}$, where $s=(s_1,s_2)$ is the index of the weighted Sobolev space, satisfies the same hypotheses as $L$.

In the following, we establish some lower bounds for the operator $\mathcal{D}+A_0+A_1$. The symbol $d(t,x,\xi)$ of the operator $\mathcal{D}(t)+\mathcal{D}^*(t)$ is such that
\begin{linenomath*}
\[
	d \in G^{0,0}\{0,0;0\}_{int,2}(\omega,g_{\Phi,k}) + G^{0,0}\{0,0,1,1;\delta_1\}_{ext,1}(\omega,g_{\Phi,k}).
\]
\end{linenomath*}
 As $t^{1-\varepsilon}d \in C([0,T];G^{0,0}(\omega,g_{\Phi,k}^{(1,\delta),(1-\delta,0)})$, it follows that
\begin{equation}\label{lb1}
2\Re \la \mathcal{D}{U},{U} \ra_{L^2} \geq -\frac{C}{t^{1-\varepsilon}} \la {U},{U} \ra_{L^2}, \quad C>0.
\end{equation}

We perform a change of variable, which allows us to control lower order terms. We set 
\begin{equation}\label{c1}
	V_1(t,x)= e^{-\int_{0}^{t}\tilde \psi(r,x,D_x)dr}{U}(t,x),
\end{equation}
where $ \tilde \psi(t,x,\xi)$ is as in (\ref{tpsi}). From (\ref{a6}), (\ref{a7}) and (\ref{a8}), we observe that the operator $e^{\pm\int_{0}^{t}\tilde \psi(r,x,D_x)dr}$ is a finite order pseudodifferential operator. The equation (\ref{c1}) implies that
\begin{linenomath*}
\[
	e^{\int_{0}^{t}\tilde \psi(r,x,D_x)dr}V_1(t,x) = (1+K_2^{(1)}(t,x,D_x))U(t,x),
\]
\end{linenomath*}
where we used Proposition \ref{prop3}, Lemma \ref{lem5}, Lemma \ref{conju} and Remark \ref{balance} to note that
\begin{linenomath*}
\[
    \begin{aligned}
		e^{\int_{0}^{t}\tilde \psi(r,x,D_x)dr} e^{-\int_{0}^{t}\tilde \psi(r,x,D_x)dr} & = I + K_2^{(1)}(t,x,D_x),\\
		e^{-\int_{0}^{t}\tilde \psi(r,x,D_x)dr} e^{\int_{0}^{t}\tilde \psi(r,x,D_x)dr} & = I + K_2^{(2)}(t,x,D_x),
	\end{aligned}
\]
\end{linenomath*}
for $\sigma(K_2^{(j)})\in G^{(-1+\varepsilon)e}(\omega;g_{\Phi,k}), j=1,2$. We choose $k>k_1$ for large $k_1$ so that the operator norm of $K_2^{(j)}, j=1,2$ is strictly lesser than $1$ and the existence of 
\begin{equation}\label{inv1}
	(1+K_2^{(j)}(t,x,D_x))^{-1} = \sum_{l=0}^{\infty} (-1)^jK_2^{(j)}(t,x,D_x)^l, \quad j=1,2, 
\end{equation}
is guaranteed. Notice that
	\begin{equation}
	\label{IstCV}
		\begin{aligned}
			V_1(0,x) & = U(0,x),\\
			\Vert{U}(t,\cdot)\Vert_{\Phi;-\kappa_0e} &\leq 2\Vert V_1(t,\cdot)\Vert_{L^2} \text{ , } \kappa_0>0 	\text{ , }0<t\leq T, \\
			U(t,x) &= (1+K_2^{(1)}(t,x,D_x))^{-1} e^{\int_{0}^{t}\tilde \psi(r,x,D_x)dr}V_1(t,x).
		\end{aligned}
	\end{equation}
Here $\kappa_0$ is same as $\kappa_{\alpha\beta}$ appearing in (\ref{bd}) for $\alpha=\beta=0$ . With the prescribed change of variable, we obtain that the pseudodifferential system (\ref{FOS1}) is equivalent to $L_1V_1 = F_1$  where
\begin{linenomath*}
\begin{equation*} 
	\begin{rcases}
		L_1 = \partial_t+\mathcal{D}+\tilde \psi I+A+R_1,\; A=A_0+A_1,\\
		F_1 = (1+K_2^{(2)}(t,x,D_x))^{-1}e^{-\int_{0}^{t}\tilde \psi(r,x,D_x)dr} (1+K_2^{(1)}(t,x,D_x)) F		
	\end{rcases}
\end{equation*}
\end{linenomath*}
and operator $R_1$ is such that
\begin{linenomath*}
\[
\begin{aligned}
	(\log(1+\P\japxik))^{-1} \sigma(R_1)(t,x,\xi)  &\in G^{\varepsilon,1}\{1,1;\delta_1\}_{int,6} (\omega\Phi^{-1+\varepsilon}, g_{\Phi,k}) \\
	& \qquad + 	G^{-1,1}\{1,0,1,1;\delta\}_{ext,1} (\Phi^{-1}, g_{\Phi,k}),
\end{aligned}
\]
\end{linenomath*}
in other words $t^{1-\varepsilon}(\log(1+\Phi(x) \la \xi \ran))^{-1} \sigma(R_1) \in {C([0,T];G^{0,0}(\{\Phi;g^{(1,\delta),(1-\delta,0)}_{\Phi,k}\})}.$
We refer the reader to Proposition \ref{prop3}, Lemma \ref{lem5}, Lemma \ref{conju} and Remark \ref{balance} in the Appendix for the details. By the definition of $\tilde \psi$, we observe that $\tilde \psi I + A $ satisfies 
\begin{linenomath*}
\begin{align*}
	2\tilde \psi I+\sigma(A+&A^*)\geq 0,\\
	t^{1-\varepsilon}(\tilde \psi I+\sigma(A)) &\in C([0,T];G^{\varepsilon',1}(\Phi^{\varepsilon'}, g_{\Phi,k}^{(1,\delta),(1-\delta,0)})).
\end{align*}
\end{linenomath*}
We now apply sharp G\r{a}rding inequality (see \cite[Theorem 18.6.14]{Horm} to the operators $2\tilde{\psi}_0I+A_0$ and $2\tilde{\psi}_1I+A_1$ separately. The symbols of these operators are governed by the metrics $g_{\Phi,k}$  and $g_{\Phi,k}^{(1,\delta),(1-\delta,0)}$ respectively where the respective Planck functions are $h(x,\xi)=(\P \la \xi \rak)^{-1}$ and $\tilde h(x,\xi)=\Phi(x)^{-1+\delta} \la \xi \rak ^{-1+\delta}.$ Notice that the symbol $\sigma(A_0)$ has the weight function $\o\japxik$ while the Planck function of the governing metric is given by $ h(x,\xi)$. Hence, for the application of sharp G\r{a}rding inequality, we need 
\begin{linenomath*}
 \[
 	\o \leq \P.
 \]
 \end{linenomath*}
 Ensuring this yields 
\begin{equation}
\label{lb2}
2\Re \la (\psi I+A)V_1, V_1\ra_{L^2} \geq -Ct^{-1+\varepsilon}\langle V_1,V_1 \rangle_{L^2} \text{ , }C>0.
\end{equation}

As for the operator $R_1$, since the symbol $t^{1-\varepsilon}(\log(1+\Phi(x) \langle \xi \rangle))^{-1}R_1 $ is uniformly bounded, for a large choice of $\kappa_1$, the application of sharp G\r{a}rding inequality yields		
\begin{equation}
\label{R1}
2 \Re \langle R_1V_1,V_1 \rangle_{L^2} \geq -\frac{\kappa_1}{t^{1-\varepsilon}} \big(2 \Re \la \log(1+\Phi(x) \la D_x \rak)V_1,V_1 \ra_{L^2} + \Vert V_1\Vert_{L^2}\big).
\end{equation}
We make a further change of variable
\begin{linenomath*}
\[
V_2(t,x)=(1+\Phi(x) \la D_x \rak)^{-\mu(t)}V_1(t,x), \qquad  \mu(t)=\kappa_1t^\varepsilon/\varepsilon,
\] 
\end{linenomath*}
where $\kappa_1$ is the constant as in (\ref{R1}). Implying
\begin{linenomath*}
\[
	(1+\Phi(x) \la D_x \rak)^{\mu(t)} V_2(t,x) = (1+K_3(t,x,D_x))V_1(t,x),
\]
\end{linenomath*}
where $\sigma(K_3) \in C([0,T]; G^{-1,-1}(\omega, g_{\Phi,k}) )$. As in (\ref{inv1}), we choose $k>k_2$, $k_2$ large, so that $(I+K_3(t,x,D_x))^{-1}$ exists. From now on we fix $k$ such that $k>k^*=\max\{k_1,k_2\}$. Further, note that
\begin{linenomath*}
\begin{equation*}
	\begin{aligned}
		V_2(0,x) &= U(0,x), \\
		\Vert{U}(t,\cdot)\Vert_{\Phi;-\Lambda(t) e} &\leq 2^{\mu(T)+1}\Vert V_2(t,\cdot)\Vert_{L^2}, \quad \Lambda(t)=\kappa_0+\kappa_1t^\varepsilon/\varepsilon \text{ , } 0 < t \leq T,\\
		V_1(t,x) &= (1+K_3(t,x,D_x))^{-1} (1+\Phi(x) \la D_x \rak)^{\mu(t)}V_2(t,x).
    \end{aligned}
\end{equation*}
\end{linenomath*}
This implies that $L{U}=F$ if and only if $L_2V_2=F_2$ where
	\begin{equation}\label{last}
		\begin{rcases}
		L_2= \partial_t + \mathcal{D}+(\tilde \psi I+A)+(\kappa_1t^{-1+\varepsilon} \log (1+\Phi(x) \la 	D_x\ran)+R_1)+R_2 \\
		F_2=(1+\Phi(x) \langle D_x \rangle)^{-\mu(t)} (1+K_3(t,x,D_x)) F_1
		\end{rcases}
	\end{equation}
 and the operator $R_2$ is such that
 \begin{linenomath*}
\[
\begin{aligned}
	\sigma(R_2)(t,x,\xi)  &\in G^{0,0}\{1,1;\delta_1\}_{int,6} (\Phi, g_{\Phi,k}) + 	G^{0,0}\{1,0,1,1;\delta\}_{ext,1} (\Phi, g_{\Phi,k}),
\end{aligned}
\]
\end{linenomath*}
in other words $t^{1-\varepsilon}\sigma(R_2) \in {C([0,T];G^{0,0}(\{\Phi;g^{(1,\delta),(1-\delta,0)}_{\Phi,k}\})}.$
From (\ref{lb1}), (\ref{lb2}), (\ref{R1}) and noting the fact that the operator $t^{1-\varepsilon}R_2$ is uniformly bounded in $L^2(\mathbb{R}^n)$ for $0 \leq t \leq T$, it follows that
	\begin{equation}
		\label{K}
		2 \Re \langle \mathcal{K}V_2,V_2 \rangle_{L^2} \geq -\frac{C}{t^{1-\varepsilon}}\la V_2,V_2 \ra_{L^2}, \quad C>0,
	\end{equation}
where $\mathcal{K} = \mathcal{D}+(\psi I+A)+(\kappa_1t^{-1+\varepsilon} \log (1+\Phi(x) \la 	D_x\ran)+R_1)+R_2$.
From (\ref{last}) and (\ref{K}), we have
\begin{linenomath*}
\[
\partial_t \Vert V_2 \Vert_{L^2}^2 \leq C(t^{-1+\varepsilon} \Vert V_2 \Vert_{L^2}^2 + \Vert F_2 \Vert_{L^2}^2).
\]
\end{linenomath*}
We apply Gronwall's lemma to obtain
\begin{linenomath*}
\[
\Vert  V_2(t,\cdot)\Vert _{L^2}^2\leq e^{Ct^\varepsilon/\varepsilon} \Vert  V_2(0,\cdot)\Vert ^2_{L^2},
\]
\end{linenomath*}
for $t \in [0,T]$. In other words, 
\begin{linenomath*}
\[
\Vert {U}(t,\cdot)\Vert_{\Phi;s-\Lambda(t) e}^2\leq 4^{\alpha(T)+1}e^{CT^\varepsilon/\varepsilon} \Vert {U}(0,\cdot)\Vert ^2_{\Phi,s} .
\]	
\end{linenomath*}
 Since the above inequality is true for any $\varepsilon \in (0,1-\delta),$ taking $\varepsilon \to 1-\delta$ we obtain
 \begin{linenomath*}
 \[
 \Vert {U}(t,\cdot)\Vert_{\Phi;s-\Lambda(t) e}^2\leq 4^{\alpha(T)+1}e^{C\frac{T^{1-\delta}}{1-\delta}} \Vert {U}(0,\cdot)\Vert ^2_{\Phi,s} .
 \]	
 \end{linenomath*}
  Returning to our original solution $u=u(t,x)$, we obtain that
  \begin{linenomath*}
\[
\sum_{j=0}^{1} \Vert \partial_t^ju(t,\cdot) \Vert_{\Phi;s+(1-j-\Lambda(t))e} \; \leq  e^{C\frac{T^{1-\delta}}{1-\delta}} \Bigg(\sum_{j=1}^{2} \Vert f_j\Vert_{\Phi;s+(2-j)e} 
+ \int_{0}^{t}\Vert f(\tau,\cdot)\Vert_{\Phi;s-\Lambda(\tau)e}\;d\tau\Bigg)
\]
\end{linenomath*}
for $0 \leq t \leq T$. This means that the original problem (\ref{eq1}) is well-posed for $u=u(t,x)$, with 
\begin{linenomath*}
\[
u \in C([0,T];H^{s+e(1-\Lambda(t))}_\Phi) \cap C^{1}([0,T];H^{s-e\Lambda(t)}_\Phi).
\]
\end{linenomath*}
This concludes the proof.

\section{Cone Condition}\label{cone}
	Existence and uniqueness follow from the a priori estimate established in the previous section. It now remains to prove the existence of cone of dependence. 
	
	We note here that the $L^1$ integrability of the logarithmic singularity plays a crucial in arriving at the finite propagation speed. The implications of the discussion in \cite[Section 2.3 \& 2.5]{JR} to the global setting suggest that if the Cauchy data in (\ref{eq1}) is such that $f \equiv 0$ and $f_1,f_2$ are supported in the ball $\vert x \vert \leq R$, then the solution to Cauchy problem (\ref{eq1}) is supported in the ball $\vert x \vert \leq R+\gamma \o (\log(1+1/t))t$. Observe that the quantity $t\log(1+1/t)$ is bounded in $[0,T]$. The constant $\gamma$ is such that the quantity $\gamma \o \log(1+1/t)$ dominates the characteristic roots, i.e.,
	\begin{equation}
		\label{speed}
		\gamma= \sup\Big\{\sqrt{a(t,x,\xi)}\o^{-1} (\log(1+1/t))^{-1}:(t,x,\xi) \in[0,T] \times \R^n_x \times \R^n_\xi,\:|\xi|=1\Big\}.
	\end{equation}
Note that the support of the solution increases as $|x|$ increases since $\o$ is monotone increasing function of $|x|$. 

In the following we prove the cone condition for the Cauchy problem $(\ref{eq1})$. Let $K(x^0,t^0)$ denote the cone with the vertex $(x^0,t^0)$:
\begin{linenomath*}
	\[
	K(x^0,t^0)= \{(t,x) \in [0,T] \times \R^n : |x-x^0| \leq \gamma\o \log(1+1/(t^0-t))(t^0-t)\}.
	\]
\end{linenomath*}
Observe that the slope of the cone is anisotropic, that is, it varies with both $x$ and $t$.

\begin{prop}
	The Cauchy problem (\ref{eq1}) has a cone dependence, that is, if
		\begin{equation}\label{cone1}
			f\big|_{K(x^0,t^0)}=0, \quad f_i\big|_{K(x^0,t^0) \cap \{t=0\}}=0, \; i=1, 2,
		\end{equation}
	then
		\begin{equation}\label{cone2}
			u\big|_{K(x^0,t^0)}=0.
		\end{equation}
\end{prop}
\begin{proof}
	Consider $t^0>0$, $\gamma>0$ and assume that  (\ref{cone1}) holds. We define a set of operators $P_\varepsilon(t,x,\partial_t,D_x), 0 \leq \varepsilon \leq \varepsilon_0$ by means of the operator $P(t,x,\partial_t,D_x)$ in (\ref{eq1}) as follows
	\begin{linenomath*}
		\[
		P_\varepsilon(t,x,\partial_t,D_x) = P(t+\varepsilon,x,\partial_t,D_x), \: t \in [0,T-\varepsilon_0], x \in \R^n,
		\]
	\end{linenomath*}
	and $\varepsilon_0 < T-t^0$, for a fixed and sufficiently small $\varepsilon_0$. For these operators we consider Cauchy problems
	\begin{linenomath*}
		\begin{alignat}{2}
			P_\varepsilon v_\varepsilon & =f,  &&  t \in [0,T-\varepsilon_0], \; x \in \R^n,\\
			\partial_t^{k-1}v_\varepsilon(0,x)& =f_k(x),\qquad && k=1,2.
		\end{alignat}
	\end{linenomath*}
	Note that $v_\varepsilon(t,x)=0$ in $K(x^0,t^0)$ and $v_\varepsilon$ satisfies the a priori estimate (\ref{est2}) for all $t \in[0,T-\varepsilon_0]$. Further, we have 
	\begin{linenomath*}
		\begin{alignat}{2}
			P_{\varepsilon_1} (v_{\varepsilon_1}-v_{\varepsilon_2}) & = (P_{\varepsilon_2}-P_{\varepsilon_1})v_{\varepsilon_2},\qquad  &&  t \in [0,T-\varepsilon_0], \; x \in \R^n,\\
			\partial_t^{k-1}(v_{\varepsilon_1}-v_{\varepsilon_2})(0,x)& = 0,\qquad && k=1,2.
		\end{alignat}
	\end{linenomath*}
	Since our operator is of second order, for the sake of simplicity we denote $b_{j}(t,x)$, the coefficients of lower order terms, as $a_{0,j}(t,x), 1 \leq j \leq n,$ and $b_{n+1}(t,x)$ as $a_{0,0}(t,x)$. Let $a_{i,0}(t,x) =0, \; 1 \leq i \leq n.$ Substituting $s-e$ for $s$ in the a priori estimate, we obtain
		\begin{equation}\label{cone3}
			\begin{aligned}
				&\sum_{j=0}^{1} \Vert  \partial_t^j(v_{\varepsilon_1}-v_{\varepsilon_2})(t,\cdot) \Vert_{\Phi,k;s-(j+\Lambda(t))e} \\
				&\leq C \int_{0}^{t}\Vert (P_{\varepsilon_2}-P_{\varepsilon_1})v_{\varepsilon_2}(\tau,\cdot)\Vert_{\Phi,k;s-e-\Lambda(\tau)e}\;d\tau\\
				&\leq C \int_{0}^{t} \sum_{i,j=0 }^{n}\Vert (a_{i,j}(\tau+\varepsilon_1,x) - a_{i,j}(\tau+\varepsilon_2,x))D_{ij} v_{\varepsilon_2}(\tau,\cdot)\Vert_{\Phi,k;s-e-\Lambda(\tau)e}\;d\tau,
			\end{aligned}
		\end{equation}
	where $D_{00}=I, D_{i0}=0, i\neq 0, D_{0j}=\partial_{x_j},j \neq 0$ and $D_{ij}=\partial_{x_i} \partial_{x_j}, i,j \neq 0$. Using the Taylor series approximation in $\tau$ variable, we have
	\begin{linenomath*}
		\begin{align*}
			|a_{i,j}(\tau+\varepsilon_1,x) - a_{i,j}(\tau+\varepsilon_2,x)| &= \Big|\int_{\tau+\varepsilon_2}^{\tau+\varepsilon_1} (\partial_ta_{i,j})(r,x)dr \Big|\\
			&\leq \o^{2} \Big|\int_{\tau+\varepsilon_2}^{\tau+\varepsilon_1}\frac{dr}{r}\Big|\\
			&\leq \o^{2}|E(\tau,\varepsilon_1,\varepsilon_2)|,
		\end{align*}
	\end{linenomath*}
	where
	\begin{linenomath*}
		\[
		E(\tau,\varepsilon_1,\varepsilon_2) =		
			\log \Bigg(1+\frac{\varepsilon_1-\varepsilon_2}{\tau+\varepsilon_2}\Bigg) . 
		\]
	\end{linenomath*}
	Note that $\o \leq \P$ and $E(\tau,\varepsilon,\varepsilon)=0$.
	Then right-hand side of the inequality in (\ref{cone3}) is dominated by
	\begin{linenomath*}
		\begin{equation*}\label{cone4}
			C \int_{0}^{t} |E(\tau,\varepsilon_1,\varepsilon_2)| \Vert  v_{\varepsilon_2}(\tau,\cdot)\Vert_{\Phi,k;s+(1-\Lambda(\tau))e}\;d\tau,
		\end{equation*}
	\end{linenomath*}
	where $C$ is independent of $\varepsilon$. By definition, $E$ is $L_1$-integrable in $\tau$.
	
	The sequence $v_{\varepsilon_k}$, $k=1,2,\dots$ corresponding to the sequence $\varepsilon_k \to 0$ is in the space
	\begin{linenomath*}
		\[
		C\Big([0,T^*];H^{s-\Lambda(t))e,}_{\Phi,k}\Big) \bigcap C^{1}\Big([0,T^*];H^{s-e-\Lambda(t)e,}_{\Phi,k}\Big), \quad T^*>0,
		\]
	\end{linenomath*}
	and $u=\lim\limits_{k\to\infty}v_{\varepsilon_k}$ in the above space and hence, in $\mathcal{D}'(K(x^0,t^0))$. In particular,
	\begin{linenomath*}
		\[
		\la u,\varphi\ra = \lim\limits_{k\to\infty} \la v_{\varepsilon_k}, \varphi \ra =0,\; \forall \varphi \in \mathcal{D}(K(x^0,t^0))
		\]
	\end{linenomath*}
	gives (\ref{cone2}) and completes the theorem. 
\end{proof}

\section{Counterexample}\label{CE}
    
 We consider a Cauchy problem of the form
	\begin{equation}
		\label{ce}
		\begin{aligned}
			\partial_t^2 u(t,x) + c&(t)A(x,D_x)u=0, \quad (t,x) \in [0,T] \times \R^n,\\
			u(0,x) = 0, &\quad \partial_tu(0,x) = f(x),
		\end{aligned}
	\end{equation}
	where $A(x,D_x) = \japx(kI-\triangle_x)\japx$ is a G-elliptic, positive, self-adjoint operator with the domain $D(A)=\{u \in L^2(\R^n): Au \in L^2(\R^n)\}$ and $c \in \mathcal{C}(\mu_1,\mu_2,\theta)$, which is defined below. 
	
	In order to show that there exists a function $c(t) \in \mathcal{C}(\mu_1,\mu_2,\theta)$ for which the Cauchy problem (\ref{ce}) has infinite loss of decay and derivatives, we extend the techniques developed by Ghisi and Gobbino \cite[Section 4]{GG} to a global setting. In our work, we indicate the amount of loss using an infinite order pseudodifferential operator as given in Remark \ref{lossop}. 
	
	\begin{defn}\label{ps}
		We denote $\mathcal{C} \left(\mu_{1}, \mu_{2}, \theta \right)$ as the set of functions $c \in C \left(\left[0, T\right]\right) \cap C^{1}\left(\left(0, T\right]\right)$ that satisfy the following growth estimates 
		\begin{align}
		\label{e1}
		0<\mu_1& \leq c(t) \leq  \mu_2, \quad t \in [0,T],\\
		\label{e2}
	    \vert c'(t)\vert &\leq C \frac{\theta(t)}{t}, \qquad t \in (0,T],
	\end{align}
	for a monotone decreasing function $\theta:(0,T] \to (0,+\infty)$ satisfying
	\begin{equation}
		\label{e3}
		\lim_{t\to 0^+} \theta(t) = +\infty.
	\end{equation}	
	\end{defn}
	The set $\mathcal{C} \left(\mu_{1}, \mu_{2}, \theta \right)$ is a complete metric space with respect to the metric
	\begin{linenomath*}
	\[
		d_{\mathcal{C}} \left(c_{1}, c_{2}\right):=\sup _{t \in\left(0, T\right)}\left|c_{1}(t)-c_{2}(t)\right|+\sup _{t \in\left(0, T\right)}\left\{\frac{t^{2}}{\theta(t)}\left|c_{1}^{\prime}(t)-c_{2}^{\prime}(t)\right|\right\}.
	\]
	\end{linenomath*}
	A sequence $c_{n}$ converges to $c_{\infty}$ with respect to $d_{\mathcal{C}}$ if and only if $c_{n} \rightarrow c_{\infty}$ uniformly in $\left[0, T\right]$, and for every $\tau \in\left(0, T\right)$,  $c_{n}^{\prime} \rightarrow c_{\infty}^{\prime}$ uniformly in $\left[\tau, T\right]$. In view of $L^1$ integrability of $c(t)$ in our setting, convergence with respect to $d_{\mathcal{C}}$ implies convergence in $L^{1}\left(\left(0, T\right)\right)$.
 
	The main aim of this section is to prove the following result.
	\begin{thm} \label{result2}
		The interior of the set of all $c \in \mathcal{C} \left(\mu_{1}, \mu_{2},\theta\right)$ for which the Cauchy problem (\ref{ce}) exhibits an infinite loss is nonempty.	
	\end{thm}
	\begin{rmk}
	    In our main result Theorem \ref{result1}, we considered (\ref{ce}) with $c \in \mathcal{C}(\mu_1,+\infty,1)$ i.e.,
	    \[
	            0 < \mu_1 \leq c(t) < +\infty, \text{ and } \quad
	            |c'(t)| \leq \frac{C}{t}, \quad t \in (0,T],
	    \]
	    and obtained well-posedness in $H^s_{\japx,k}(\R^n)$ with finite loss of decay and derivatives.
	    On the other hand, in Theorem \ref{result2} we consider unbounded $\theta$ as in (\ref{e3}) with no further assumptions on the second derivative of $c(t).$ Further note that there exists a function $c(t) \in \mathcal{C} \left(\mu_{1}, \mu_{2},\theta\right)$ with finite loss may be with additional assumptions on the higher order derivatives  (see for instance \cite{GG,KuboReis,NUKCori1}).
	\end{rmk}
	 
	Since the operator $A$ is positive and G-elliptic and the symbol
	\[
		\sigma(A^{\alpha}) \sim \japx^{2\alpha} \japxik^{2\alpha} + \textnormal{ lower order terms, }
	\]
	the Sobolev spaces associated to the Cauchy problem (\ref{ce}) are $H^{2\alpha,2\alpha}_{\japx,k},\alpha \in \R,$ defined in (\ref{Sobo}.)
    We characterize the Sobolev spaces $H^{2m,2m}_{\japx,k}(\R^n),m \in \mathbb{Z},$ using the spectral theorem \cite[Theorem 4.2.9]{nicRodi} for pseudodifferential operators on $\R^n.$  The theorem guarantees the existence of an orthonormal basis $(e_i(x))_{i=1}^{\infty}, e_i \in \mathcal{S}(\R^n),$ of $L^2(\R^n)$ and a nondecreasing sequence $(\lambda_i)_{i=1}^{\infty} $ of nonnegative real numbers diverging to $+\infty$ such that $Ae_i(x)=\lambda_i^2e_i(x).$
	
	Using $\lambda_{i}$s we identify $v(x)\in H^{2m,2m}_{\japx,k}(\R^n)$ with a sequence $(v_i)$ in weighted $ \ell^2$, where $v_i=\la v,e_i\ra_{L^2}$. One can prove the following proposition using Riesz representation theorem showing the correspondence between $H^{2m,2m}_{\japx,k}(\R^n)$ and a weighted $\ell^2$ space.
	
	\begin{prop}
		Let $(v_{i})$ be a sequence  of real numbers and $m \in \mathbb{Z}.$ Then 
		\[
		    \sum_{i=1}^{\infty}v_ie_i(x)\in H^{2m,2m}_{\japx,k}(\R^n) \quad \text{ if and only if} \quad 	\sum_{i=1}^{\infty}\lambda_{i}^{2 m} v_{i}^{2}<+\infty.
		\]
	\end{prop}

	The solution to (\ref{ce}) is $u(t,x) =\sum_{i=1}^{\infty}u_i(t)e_i(x)$ where the functions $u_i(t)$ satisfy the decoupled system of ODEs
	\begin{equation}
		\label{ce2}
		\begin{aligned}
		u_i''(t) + c(t) &\lambda_i^2u_i(t) = 0, \quad i \in \mathbb{N}, \; t \in [0,T],\\
		u_i(0)= 0, &\quad u_i'(0) = f_{i},
		\end{aligned}
	\end{equation}
	for $\partial_tu(0,x)=f(x)=\sum_{i=1}^{\infty}f_ie_i(x)$.
	
	We say that the solution to (\ref{ce}) experiences infinite loss of decay and derivatives if the initial velocity $f\in H^{2m,2m}_{\japx,k}$ for all $m \in \mathbb{Z}^+$ but $(u,\partial_tu) \notin H^{-2m+1,-2m+1}_{\japx,k} \times H^{-2m,-2m}_{\japx,k}$ for any $m \in \mathbb{Z}^+$ and for $t \in (0,T].$
	\begin{rmk}\label{lossop}
		Given that the propagation speed $c(t)$ satisfies the estimate (\ref{e2}), the methodology used in Section \ref{Proof1} in arriving at the energy estimate in fact suggests that the loss of infinite order is quite expected. The following observation
		\[
			(\log(1/t))^{-1} \int_{t}^{T} \vert c^\prime(t) \vert dt \to + \infty \; \text{ as } \; t \to 0^+,
		\]
		along with the region definitions given in Section \ref{zones} imply that the averaged behavior of the majorizing function $\tilde \psi$ in (\ref{tpsi}) is given by
		\[
			(\log(1+\P\japxik))^{-1} \int_{0}^{T} \vert \tilde \psi(t,x,\xi) \vert dt \to +\infty \; \text{ as } \; |x| + |\xi| \to \infty.
		\] 
		This suggests that the function spaces involved in the change of variable (\ref{c1}) are of exponential order in both $x$ and $D_x$. For example, if $\theta(t)=\log(1/t)$, then the function spaces are of the form
		\[
			\left\{ v \in L^2(\R^n) : e^{\kappa(\log(1+\P\la D_x \rak)^2} u \in L^2(\R^n) \right\},
		\]
		and the loss is quantified in these spaces.
		Thus, knowing the nature of $\theta(t)$ allows one to quantify the infinite loss of regularity.
	\end{rmk}

	In order to prove Theorem \ref{result2} we define a dense subset of $\mathcal{C} \left(\mu_{1}, \mu_{2}, \theta\right).$
	\begin{defn}  We call $\mathcal{D}\left(\mu_{1}, \mu_{2}\right)$ the set of functions $c:\left[0, T\right] \rightarrow\left[\mu_{1}, \mu_{2}\right]$ for which there exists two real numbers $T_{1} \in\left(0, T\right)$ and $\mu_{3} \in\left(\mu_{1}, \mu_{2}\right)$ such that $c(t)=\mu_{3}$ for every $t \in\left[0, T_{1}\right]$.
	\end{defn}
	For the sake of simplicity, let us denote $\mathcal{D}\left(\mu_{1}, \mu_{2}\right)$ and $\mathcal{C}\left(\mu_{1}, \mu_{2}, \theta\right)$ by $\mathcal{D}$ and $\mathcal{C}$, respectively. From \cite[Prposition 4.7]{GG}, $\mathcal{D} \cap \mathcal{C}$ is dense in $\mathcal{C}.$ We note here that the weight factor $\frac{t^2}{\theta(t)}$ appearing in the definition of the metric $d_{\mathcal{C}}$ plays a crucial role in proving the denseness.
	
	Following the terminology of Ghisi and Gobbino \cite{GG}, we now introduce special classes of propagation speeds: universal and asymptotic activators. Let $\phi:(0,+\infty) \rightarrow(0,+\infty)$ be a function. 
	\begin{defn} 
	A universal activator of the sequence $(\lambda_{i})$ with rate $\phi$ is a propagation speed $c \in L^{1}\left(\left(0, T\right)\right)$ such that the corresponding sequence $(u_{i}(t))$ of solutions to 
	\begin{linenomath*}
	\[
		u_i^{''}(t) + c(t) \lambda_i^2u_i(t) = 0, \quad 
		u_i(0)= 0, \quad u_i'(0) = 1,
	\]
	\end{linenomath*}
	satisfies
	\begin{linenomath*}
	$$
	\limsup _{i \rightarrow+\infty}\left(\left|u_{i}^{\prime}(t)\right|^{2}+\lambda_{i}^{2}\left|u_{i}(t)\right|^{2}\right) \exp \left(-\phi\left(\lambda_{i}\right)\right) \geq 1,  \quad \forall t \in\left(0, T\right].
	$$
	\end{linenomath*}
	\end{defn}	
	Then the solution $u$ to problem (\ref{ce}) is given by
	\begin{linenomath*}
	\[
		u(t,x) = \sum_{i=1}^{\infty} f_iu_i(t)e_i(x).
	\]
	\end{linenomath*}
	\begin{defn}
		A family of asymptotic activators with rate $\phi$ is a family of propagation speeds $\left\{c_{\lambda}(t)\right\} \subseteq L^{1}\left(\left(0, T\right)\right)$ with the property that, for every $\delta \in\left(0, T\right),$ there exist two positive constants $M_{\delta}$ and $\lambda_{\delta}$ such that the corresponding family $\left\{u_{\lambda}(t)\right\}$ of solutions to 
		\begin{linenomath*}
		\[
			u_{\lambda}''(t) + c_{\lambda}(t) \lambda^2u_{\lambda}(t) = 0, \quad 
			u_{\lambda}(0)= 0, \quad u_{\lambda}'(0) = 1,
		\]
		\end{linenomath*}
		satisfies
		\begin{linenomath*}
		$$
			\qquad\left|u_{\lambda}^{\prime}(t)\right|^{2}+\lambda^{2}\left|u_{\lambda}(t)\right|^{2} \geq M_{\delta} \exp (2 \phi(\lambda)), \quad \forall t \in\left[\delta, T\right], \quad \forall \lambda \geq \lambda_{\delta}.
		$$
		\end{linenomath*}
	\end{defn}
	Due to denseness of $\mathcal{D}$ in $\mathcal{C}$, for every $c \in \mathcal{D}$ there exists a family of asymptotic activators $(c_{\lambda}) \subseteq \mathcal{C}$ with rate $\phi$ such that $c_{\lambda} \rightarrow c$ with respect to $d_{\mathcal{C}} .$
	The existence of families of asymptotic activators converging to elements of a dense set implies the existence of a residual set of  universal activators. Since the problem (\ref{ce}) exhibits an infinite loss of derivatives and decay whenever $c(t)$ is a universal activator, construction of couterexample amounts to the existence of such asymptotic activators.
	Once the asymptotic activators are constructed and if $\phi$ is such that $\phi(\lambda) \rightarrow+\infty$ as $\lambda \rightarrow+\infty$, then Proposition 4.5 in \cite{GG} guarantees that the set of elements in $\mathcal{C}$ that are universal activators of the sequence $(\lambda_{i})$ with rate $\phi$ is residual in $\mathcal{C}.$ In addition, if the function $\phi$ is such that
	\begin{equation}\label{logphi}
		\lim _{\lambda \rightarrow+\infty} \frac{\phi(\lambda)}{\log \lambda}=+\infty,
	\end{equation}
	then by Proposition 4.3 in \cite{GG} one can show that for each of the universal activator $c(t)$ of sequence $(\lambda_n)$ with rate $\phi$ the solution to problem (\ref{ce}) exhibits infinite loss of regularity.
	
	We are left with construction of asymptotic activators with rate $\phi$ satisfying (\ref{logphi}). We consider $T_{1}$ and $ \gamma$ such that $0<T_{1}<T$ and
	$0<\mu_{1}<\gamma^{2}<\mu_{2},$ and define a initially constant speed $c_{*}:\left[0, T\right] \rightarrow\left[\mu_{1}, \mu_{2}\right]$ such that
	\begin{linenomath*}
	$$
	c_{*}(t)=\gamma^{2},  \quad \forall t \in\left[0, T_{1}\right].
	$$
	\end{linenomath*}
	For every large enough real number $\lambda,$ let $a_{\lambda}$ and $b_{\lambda}$ be real numbers such that
	\begin{linenomath*}
	\[
		a_{\lambda}:=\frac{2 \pi}{\gamma \lambda}\left\lfloor\lambda^{1 / 4}\right\rfloor, \quad \quad b_{\lambda}:=\frac{2 \pi}{\gamma \lambda}\left\lfloor\lambda^{1 / 2}\right\rfloor.
	\]
	\end{linenomath*}
	where $\lfloor \alpha \rfloor$ stands for integer part of a real number $\alpha.$ Observe that 
	\begin{linenomath*}
	$$
	0<a_{\lambda}<2 a_{\lambda}<\frac{b_{\lambda}}{2}<b_{\lambda}<T_{1}, \qquad
	\frac{\gamma \lambda a_{\lambda}}{2 \pi} \in \mathbb{N} \quad \text { and } \quad \frac{\gamma \lambda b_{\lambda}}{2 \pi} \in \mathbb{N}.
	$$
	\end{linenomath*}
	Let us choose a cutoff function $\nu: \mathbb{R} \rightarrow \mathbb{R}$ of class $C^{\infty}$ such that $0 \leq \nu(r)\leq 1$, $\nu(r)=0 $ for $r\leq 0$ and $\nu(r)=1$ for $r\geq1$. Setting 
	$\theta_{\lambda}:=\min \left\{\theta\left(b_{\lambda}\right), \log \lambda\right\}$, we define $\varepsilon_{\lambda}:\left[0, T\right] \rightarrow \mathbb{R}$ as
	\begin{linenomath*}
	$$
	\varepsilon_{\lambda}(t):=\left\{\begin{array}{ll}
		0 & \text { if } t \in\left[0, a_{\lambda}\right] \cup\left[b_{\lambda}, T_{0}\right] \\
		\frac{\theta_{\lambda}}{t} & \text { if } t \in\left[2 a_{\lambda}, b_{\lambda} / 2\right] \\
		\frac{\theta_{\lambda}}{t} \cdot \nu\left(\frac{t-a_{\lambda}}{a_{\lambda}}\right) & \text { if } t \in\left[a_{\lambda}, 2 a_{\lambda}\right] \\
		\frac{\theta_{\lambda}}{t} \cdot \nu\left(\frac{2\left(b_{\lambda}-t\right)}{b_{\lambda}}\right) & \text { if } t \in\left[b_{\lambda} / 2, b_{\lambda}\right]
	\end{array}\right.
	$$
	\end{linenomath*}
	Using the functions $c_*(t)$ and $\varepsilon_{\lambda}(t)$ we define $c_{\lambda}:\left[0, T\right] \rightarrow \mathbb{R}$ as
	\begin{linenomath*}
	$$
	c_{\lambda}(t) := c_{*}(t)-\frac{\varepsilon_{\lambda}(t)}{4 \gamma \lambda} \sin (2 \gamma \lambda t)-\frac{\varepsilon_{\lambda}^{\prime}(t)}{8 \gamma^{2} \lambda^{2}} \sin ^{2}(\gamma \lambda t)-\frac{\varepsilon_{\lambda}(t)^{2}}{64 \gamma^{4} \lambda^{2}} \sin ^{4}(\gamma \lambda t).
	$$
	\end{linenomath*}
	By Proposition 4.8 and 4.9 in \cite{GG}, $(c_\lambda(t))$ is a family of asymptotic activators with rate
	\begin{linenomath*}
	$$
	\phi(\lambda):=\frac{\theta_{\lambda}}{32 \gamma^{2}} \log \left(\frac{\left\lfloor\lambda^{1 / 2}\right\rfloor}{\left\lfloor\lambda^{1 / 4}\right\rfloor}\right),
	$$
	\end{linenomath*}
	and 
	\begin{linenomath*}
	\[
		\lim _{\lambda \rightarrow+\infty} d_{\mathcal{C}}\left(c_{\lambda}, c_{*}\right)=0.
	\]
	\end{linenomath*}
	 Since $c_{*}$ is a generic element of a dense subset,  we see that these universal activators cause an infinite loss of decay and derivatives.
	
	\addcontentsline{toc}{section}{Acknowledgements}
	\section*{Acknowledgements}
		The first author is funded by the University Grants Commission, Government of India, under its JRF and SRF schemes.

	\addcontentsline{toc}{section}{Appendix}
	\section*{Appendix}
	In this Section, we give a calculus for the parameter dependent global symbol classes defined in Section \ref{Symbol classes}. The following two propositions give their relations to the symbol classes $G^{m_1,m_2}(\omega,g_{\Phi,k}^{\rho,r}).$ 
	\begin{prop}\label{prop1}
		Let $a=a(t, x, \xi)$ be a symbol with
		\begin{linenomath*}
		$$
			a \in G^{0,0}\{0,0;0\}_{int,N}\wm +G^{m_1,m_2}\{l_1,l_2,l_3,l_4;\delta\}_{ext,N}\wm.
		$$
		\end{linenomath*}
		Then, for $\tilde{m}_1=\max\{0,m_1\}$, $\tilde{m}_2=\max\{0,m_2\}$ and for any $\varepsilon>0$, 
		\begin{linenomath*}
		$$
		t^{l_1+\delta l_2}a \in C\left([0, T] ; G^{\tilde m_1+\varepsilon,1} \wmt \right).
		$$
		\end{linenomath*}
	\end{prop}
	\begin{proof}
		The hypothesis of the proposition implies that
		\begin{equation}\label{a1}
		\begin{aligned}
			\left|D_{x}^{\beta} \partial_{\xi}^{\alpha} a(t, x, \xi)\right| & \leq C_{\alpha, \beta}\japxi^{-|\alpha|}\P^{-|\beta|}+C_{\alpha, \beta}\japxi^{m_{1}-|\alpha|} \o^{m_2}\P^{-|\beta|} \\
			& \qquad \times t^{-(l_1+\delta(l_2+|\beta|))} (\log(1+1/t))^{l_3+l_4(|\alpha|+|\beta|)}.
		\end{aligned}
		\end{equation}
		From the definition of the regions, one can observe that for any $\varepsilon>0$
		\begin{equation}\label{a2}
		\begin{aligned}
			(\log(1+1/t))^{l_3+l_4(|\alpha|+|\beta|)} &\leq (\P\japxik)^{\varepsilon}, \\
			t^{-\delta|\beta|} & \leq (\P\japxik)^{\delta|\beta|}.
		\end{aligned}
	\end{equation}
		Hence,
		\begin{linenomath*}
		$$
			t^{l_1+\delta l_2}\left|D_{x}^{\beta} \partial_{\xi}^{\alpha} a(t, x, \xi)\right|  \leq C_{\alpha, \beta}\japxi^{\tilde m_{1}+\varepsilon-|\alpha|+\delta|\beta|} \o^{\tilde m_2}\P^{\varepsilon-(1-\delta)|\beta|}.
		$$
		\end{linenomath*}
	\end{proof}
	\begin{rmk}
		Consider $\varepsilon,\varepsilon'>0$ such that $\varepsilon<\varepsilon'<1-\delta$. Let 
		\begin{linenomath*}
		\[
			\begin{aligned}
				a &\in G^{0,0}\{0,0;0\}_{int,N}\wm +G^{0,0}\{1,0,l_3,l_4;\delta\}_{ext,N}\wm, \\
				b &\in G^{0,0}\{0,0;0\}_{int,N}\wm +G^{0,0}\{0,1,l_3,l_4;\delta\}_{ext,N}\wm.
			\end{aligned}
		\]
		\end{linenomath*}
		Then, from (\ref{a2}), we have
		\begin{linenomath*}
		\[
			\begin{aligned}
				a &\in C\left([0, T] ; G^{\varepsilon',1} \wmtt \right), \\
				b &\in C\left([0, T] ; G^{\varepsilon,1} (\Phi^{\varepsilon}, g_{\Phi,k}^{(1,\delta),(1-\delta,0)}) \right).
			\end{aligned}
		\]
		\end{linenomath*}
	\end{rmk}
	
	Let us consider an indicator function ${\bf{I}}_r: [0,\infty) \to \{0,1\}$ defined as 
	\begin{linenomath*}
	\[
	{\bf{I}}_r = \begin{cases}
		0, & \text{ if } r = 0\\
		1, & \text{ otherwise }
	\end{cases}
	\]
	\end{linenomath*}
	and denote $1-{\bf{I}}_r$ as ${\bf{I}}_r^c.$
	
	\begin{prop}\label{prop2}
		Let $a=a(t, x, \xi)$ be a symbol with
		\begin{linenomath*}
		$$
			a \in G^{m_1,m_2}\{\tilde l_1,\tilde l_2;\delta\}_{int,N} \wm + G^{0,0}\{0,0,0,0;0\}_{ext,N}\wm, 
		$$
		\end{linenomath*}
		and let 
		\begin{linenomath*}
		\[
			\tilde l = \begin{cases}
				l_2\delta & \text{ if }{\tilde l_2>0}\\
				\varepsilon & \text{ if } \tilde l_2 \leq 0 \text{ and } \tilde l_1>0\\
				0 & \text{ if } \tilde l_1<0 \text{ and } \tilde l_2<0.
			\end{cases}
		\]
		\end{linenomath*}
		Then we have $t^{\tilde l}a \in C\left( [0,T] ; G^{\tilde m_1,\tilde m_2} \wm \right)$ for $\tilde{m}_1=\max\{0,m_1\}$, $\tilde{m}_2=\max\{0,m_2\}$ and for any $\varepsilon>0$.
	\end{prop} 
	\begin{proof}
		The proposition follows by observing the following estimate
		\begin{equation}\label{a3}
			\begin{aligned}
				\left|D_{x}^{\beta} \partial_{\xi}^{\alpha} a(t, x, \xi)\right| & \leq C_{\alpha, \beta}\japxi^{m_{1}-|\alpha|} \o^{m_2}\P^{-|\beta|}(\log(1+1/t))^{\tilde l_1{\bf{I}}_{|\beta|}^c}t^{-\tilde l_2{\bf{I}}_{|\beta|}}  \\
				&\qquad+ C_{\alpha, \beta} \japxi^{-|\alpha|}\P^{-|\beta|}.
			\end{aligned}
		\end{equation}
	\end{proof}
	\begin{rmk}
		Suppose $a \in G^{m_1,m_2}\{1,1;\delta\}_{int,N} + G^{0,0}\{0,0,0,0;0\}_{ext,N}\wm $. Then for any $\varepsilon>0$ satisfying $\varepsilon<1-\delta$, we have $t^{1-\varepsilon}a \in G^{\tilde m_1,\tilde m_2}\wm.$
	\end{rmk}	
	
	For $\mu>0$, we set
	\begin{linenomath*}
	\[
		Q_{r,\mu} = \{(x,\xi) \in \R^{2n} : \P^{\mu} <r, \japxik^{\mu} < r\}, \qquad Q_{r,\mu}^c = \R^{2n} \setminus Q_{r,\mu}.
	\]
	\end{linenomath*}
	
	\begin{prop}\label{prop3}
		(Asymptotic expansion) Let $\{a_{j}\},j \geq 0$ be a sequence of symbols with
		\begin{linenomath*}
		$$
		\begin{aligned}
			a_{j} &\in G^{\tilde{m}_1-j,1}\{\tilde{l}_1,\tilde{l}_2;\delta_1\}_{int,N} (\omega^{\tilde{m}_2}\Phi^{-j},g_{\Phi,k}) \\
			& \qquad + G^{{m}_1-j,1}\{{l}_1+\delta_2j,{l}_2,l_3+2l_4j,l_4;\delta_2\}_{ext,N} (\omega^{{m}_2}\Phi^{-j},g_{\Phi,k}).
		\end{aligned}
		$$
		\end{linenomath*}
		Then, there is a symbol
		\begin{linenomath*}
		$$
		a \in G^{\tilde{m}_1,\tilde{m}_2}\{\tilde{l}_1,\tilde{l}_2;\delta_1\}_{int,N} \wm + G^{{m}_1,{m}_2}\{{l}_1,{l}_2,l_3,l_4;\delta_2\}_{ext,N} \wm
		$$
		\end{linenomath*}
		such that
		\begin{linenomath*}
		$$
		a(t, x, \xi) \sim \sum_{j=0}^{\infty} a_{j}(t, x, \xi)
		$$
		\end{linenomath*}
		that is for all $j_{0} \geq 1$
		\begin{linenomath*}
		$$
		\begin{aligned}
			a(t, x, \xi)-\sum_{j=0}^{j_{0}-1} a_{j}(t, x, \xi)  &\in  G^{\tilde{m}_1-j_0,1}\{\tilde{l}_1,\tilde{l}_2;\delta_1\}_{int,N} (\omega^{\tilde{m}_2}\Phi^{-j_0},g_{\Phi,k}) \\
			&\;\:+ G^{{m}_1-(1-\delta_2)j_0+\varepsilon,1}\{{l}_1,{l}_2,l_3,l_4;\delta_2\}_{ext,N} (\omega^{{m}_2}\Phi^{-(1-\delta_2)j_0+\varepsilon},g_{\Phi,k}),
		\end{aligned}
		$$
		\end{linenomath*}
		where $\varepsilon \ll 1-\delta_2$. The symbol is uniquely determined modulo $C\left((0, T] ; \mathcal{S}(\R^{2n})\right)$.
	\end{prop}
	
	\begin{proof}
		Let us fix $\varepsilon \ll 1-\delta_2$ and set $\mu=1-\delta_2-\varepsilon$. Consider a $C^{\infty}$ cut-off function, $\chi$ defined by
		\begin{linenomath*}
		$$
		\chi(x,\xi)=\left\{\begin{array}{ll}
			1, & (x,\xi) \in Q_{1,\mu} \\
			0, & (x,\xi) \in Q_{2,\mu}^c
		\end{array}\right.
		$$
		\end{linenomath*}
		and $0 \leq \chi \leq 1 .$ For a sequence of positive numbers $\varepsilon_{j} \rightarrow 0$, we define
		\begin{linenomath*}
		$$
		\begin{aligned}
			\gamma_0(x,\xi) &\equiv 1, \\
			\gamma_{{j}}(x,\xi) &=1-\chi\left(\varepsilon_{j} x,\varepsilon_{j} \xi\right), \quad j \geq 1.
		\end{aligned}
		$$
		\end{linenomath*}
		We note that $\gamma_{{j}}(x,\xi)=0$ in $Q_{1,\mu}$ for $j \geq 1.$  We choose $\varepsilon_{j}$ such that
		\begin{linenomath*}
		$$
		\varepsilon_{j} \leq 2^{-j}
		$$
		\end{linenomath*}
		and set
		\begin{linenomath*}
		$$
		a(t, x, \xi)=\sum_{j=0}^{\infty} \gamma_{{j}}(x,\xi) a_{j}(t, x, \xi).
		$$
		\end{linenomath*}
		We note that $a(t, x, \xi)$ exists (i.e. the series converges point-wise), since for any fixed point $(t, x, \xi)$ only a finite number of summands contribute to $a(t, x, \xi) .$ Indeed, for fixed $(t, x, \xi)$ we can always find a $j_{0}$ such that $\P^\mu < \frac{1}{\varepsilon_{j_0}}$,  $\japxi^\mu<\frac{1}{\varepsilon_{j_0}}$ and hence
		\begin{linenomath*}
		$$
		a(t, x, \xi)=\sum_{j=0}^{j_{0}-1} \gamma_{j}(x,\xi) a_{j}(t, x, \xi).
		$$
		\end{linenomath*}
		Observe that
		\begin{linenomath*}
		$$
		\begin{aligned}
			|D_{x}^{\beta} \partial_{\xi}^{\alpha}\left(\gamma_{{j}}(x,\xi) a_{j}(t, x, \xi)\right)|& \leq \sum\limits_{\substack{\alpha^{\prime}+\alpha^{\prime \prime}=\alpha \\ \beta^{\prime}+\beta^{\prime \prime}=\beta}}\left(\begin{array}{c}
				\alpha \\
				\alpha^{\prime}
			\end{array}\right) \left(\begin{array}{c}
				\beta \\
				\beta^{\prime}
			\end{array}\right) |\partial_{\xi}^{\alpha^{\prime}} D_x^{\beta^\prime} \gamma_{j}(x,\xi) D_{x}^{\beta^{\prime \prime}} \partial_{\xi}^{\alpha^{\prime \prime}} a_{j}(t, x, \xi)| \\
			&\leq  \mid \gamma_{{j}}(x,\xi) D_{x}^{\beta} \partial_{\xi}^{\alpha} a_{j}(t, x, \xi) \\
			&\qquad +\sum\limits_{\substack{\alpha^{\prime}+\alpha^{\prime \prime}=\alpha,|\alpha^{\prime}|>0 \\ \beta^{\prime}+\beta^{\prime \prime}=\beta,|\beta^{\prime}|>0 } }   C_{\alpha^{\prime} \beta^{\prime}}  \frac{\tilde{\chi}_{{j}}(x,\xi)}{\P^{\mu|\beta^{\prime}|} \japxik^{\mu|\alpha^{\prime}|}}D_{x}^{\beta^{\prime \prime}} \partial_{\xi}^{\alpha^{\prime \prime}} a_{j}(t, x, \xi) \mid,
		\end{aligned}
		$$
		\end{linenomath*}
		where $\tilde{\chi}_{{j}}(x,\xi)$ is a smooth cut-off function supported in $Q_{1,\mu}^c \cap Q_{2,\mu}.$
		This new cut-off function describes the support of the derivatives of $\gamma_{j}(x,\xi) .$ In the last estimate, we also used that $\frac{1}{\varepsilon_{j}}\sim \japxi^\mu$ and $\frac{1}{\varepsilon_{j}} \sim \P^\mu$ if $\tilde{\chi}_{j}(x,\xi) \neq 0 .$ We conclude that
		\begin{linenomath*}
		\[
			\begin{aligned}
				|&D_{x}^{\beta} \partial_{\xi}^{\alpha} \gamma_{j}(x,\xi) a_{j}(t, x, \xi)| \\
				& \leq \frac{1}{2^j} \Big[ \japxik^{\tilde m_1+\mu-j-|\alpha|} \o^{\tilde m_2} \P^{\mu-j-|\beta|} \left(\log(1+1/t)\right)^{\tilde l_1{\bf{I}}_{|\beta|}^c}  (1/t)^{\delta_1 \tilde l_2{\bf{I}}_{|\beta|}}  \chi(2t\P\japxik/N) \\
				& \quad+ \japxik^{m_1+\mu-j-|\alpha|} \o^{m_2} \P^{\mu-j-|\beta|} (1/t)^{l_1+\delta_2(l_2+|\beta|+j)} (\log(1+1/t))^{l_3+l_4(|\alpha| + |\beta|+2j)} \\
				& \qquad \quad \times (1- \chi(t\P\japxik/N)) \Big] \\
				& \leq \frac{1}{2^j} \Big[ \japxik^{\tilde m_1+\mu-j-|\alpha|} \o^{\tilde m_2} \P^{\mu-j-|\beta|} \left(\log(1+1/t)\right)^{\tilde l_1{\bf{I}}_{|\beta|}^c}  (1/t)^{\delta_1 \tilde l_2{\bf{I}}_{|\beta|}}  \chi(2t\P\japxik/N) \\
				& \quad+ \japxik^{m_1+\varepsilon{\bf{I}}_j+\mu-(1-\delta_2)j-|\alpha|} \o^{m_2} \P^{\mu-(1-\delta_2)j-|\beta|+\varepsilon{\bf{I}}_j} (1/t)^{l_1+\delta_2(l_2+|\beta|)}  \\
				& \qquad \quad \times (\log(1+1/t))^{l_3+l_4(|\alpha| + |\beta|)}(1- \chi(t\P\japxik/N)) \Big],
			\end{aligned}
		\]
		\end{linenomath*}
		where we have estimated $\frac{\P^\mu}{ 2^{j}} \geq 1$and $\frac{\japxik^\mu}{ 2^{j}} \geq 1$ (due to the support of cut-off functions) once in each summand and noted that in $\hyp$ for every $\varepsilon \ll 1$,
		\begin{linenomath*}
		\[	
			\begin{aligned}		
				(\log(1+1/t))^{2l_4j} & \leq(\P\japxik)^{\varepsilon{\bf{I}}_j}, \\
				(1/t)^{\delta_2j} &\leq (\P\japxik)^{\delta_2j}.
			\end{aligned}
		\]
		\end{linenomath*}
		Using this relation, we obtain
		\begin{linenomath*}
		$$
		\begin{aligned}
			|D_{x}^{\beta} &\partial_{\xi}^{\alpha} a(t,x, \xi) | \\		
			\leq &\left|D_{x}^{\beta} \partial_{\xi}^{\alpha}\left(\gamma_{0}(x,\xi) a_{0}(t, x, \xi)\right)\right| + \sum_{j=1}^{j_{0}-1}\left|D_{x}^{\beta} \partial_{\xi}^{\alpha}\left(\gamma_{j}(x,\xi) a_{j}(t, x, \xi)\right)\right| \\
			\leq& C_{\alpha\beta} \Big[ \japxik^{\tilde m_1-|\alpha|} \o^{\tilde m_2} \P^{-|\beta|} \left(\log(1+1/t)\right)^{\tilde l_1{\bf{I}}_{|\beta|}^c}  (1/t)^{\delta_1 \tilde l_2{\bf{I}}_{|\beta|}}  \chi(2t\P\japxik/N) \\
			& \quad+ \japxik^{m_1-|\alpha|} \o^{m_2} \P^{-|\beta|} (1/t)^{l_1+\delta_2(l_2+|\beta|)} (\log(1+1/t))^{l_3+l_4(|\alpha| + |\beta|)} \\
			& \qquad \quad \times (1- \chi(t\P\japxik/N)) \\
			  & \quad +\sum_{j=1}^{j_{0}-1}  \frac{1}{2^j} \Big[ \japxik^{\tilde m_1+\mu-j-|\alpha|} \o^{\tilde m_2} \P^{\mu-j-|\beta|} \left(\log(1+1/t)\right)^{\tilde l_1{\bf{I}}_{|\beta|}^c}  (1/t)^{\delta_1 \tilde l_2{\bf{I}}_{|\beta|}}   \\
			& \qquad \quad \times \chi(2t\P\japxik/N) + \japxik^{m_1+\mu-(1-\delta_2)j+\varepsilon-|\alpha|} \o^{m_2} \P^{\mu-(1-\delta_2)j+\varepsilon-|\beta|}   \\
			& \qquad \quad \times (1/t)^{l_1+\delta_2(l_2+|\beta|)}(\log(1+1/t))^{l_3+l_4(|\alpha| + |\beta|)}(1- \chi(t\P\japxik/N)) \Big]\\
			\leq & C_{\alpha\beta} \Big[ \japxik^{\tilde m_1-|\alpha|} \o^{\tilde m_2} \P^{-|\beta|} \left(\log(1+1/t)\right)^{\tilde l_1{\bf{I}}_{|\beta|}^c}  (1/t)^{\delta_1 \tilde l_2{\bf{I}}_{|\beta|}}  \chi(2t\P\japxik/N) \\
			& \quad+ \japxik^{m_1-|\alpha|} \o^{m_2} \P^{-|\beta|} (1/t)^{l_1+\delta_2(l_2+|\beta|)} (\log(1+1/t))^{l_3+l_4(|\alpha| + |\beta|)} \\
			& \qquad \quad \times (1- \chi(t\P\japxik/N)) \Big],
		\end{aligned}
		$$
		\end{linenomath*}
		where the last inequality holds by the choice $\mu.$ Thus, we have
		\begin{linenomath*}
		$$
		a \in G^{\tilde{m}_1,\tilde{m}_2}\{\tilde{l}_1,\tilde{l}_2;\delta_1\}_{int,N} \wm + G^{{m}_1,{m}_2}\{{l}_1,{l}_2,l_3,l_4;\delta_2\}_{ext,N} \wm.
		$$
		\end{linenomath*}
		Arguing as above, we have
		\begin{linenomath*}
		\[
			\begin{aligned}
					\sum_{j=j_{0}}^{\infty} \gamma_{j} a_{j} &\in 	G^{\tilde{m}_1-j_0,1}\{\tilde{l}_1,\tilde{l}_2;\delta_1\}_{int,N} (\omega^{\tilde{m}_2}\Phi^{-j_0},g_{\Phi,k}) \\
					& \qquad + G^{{m}_1-(1-\delta_2)j_0+\varepsilon,1}\{{l}_1,{l}_2,l_3,l_4;\delta_2\}_{ext,N} (\omega^{{m}_2}\Phi^{-(1-\delta_2)j_0+\varepsilon},g_{\Phi,k}),
			\end{aligned}
		\]
		\end{linenomath*}
		and thus,
		\begin{linenomath*}
		$$
		\begin{aligned}
			a(t, x, \xi)-\sum_{j=0}^{j_{0}-1} a_{j}(t, x, \xi)  \in &G^{\tilde{m}_1-j_0,1}\{\tilde{l}_1,\tilde{l}_2;\delta_1\}_{int,N} (\omega^{\tilde{m}_2}\Phi^{-j_0},g_{\Phi,k}) \\
			& + G^{{m}_1-(1-\delta_2)j_0+\varepsilon,1}\{{l}_1,{l}_2,l_3,l_4;\delta_2\}_{ext,N} (\omega^{{m}_2}\Phi^{-(1-\delta_2)j_0+\varepsilon},g_{\Phi,k}).
		\end{aligned}
		$$
		\end{linenomath*}
		Lastly, we use Proposition \ref{prop1} and \ref{prop2} to conclude that 
		\begin{linenomath*}
		$$
			 t^{l}a_{j} \in C\left([0, T] ; G^{ m_1^*-(1-\delta_2)j+\varepsilon{\bf{I}}_j,1}(\omega^{ m_2^*}\Phi^{-(1-\delta_2)j+\varepsilon{\bf{I}}_j},g_{\Phi,k}^{(1,\delta),(1-\delta,0)}) \right)
		$$
		\end{linenomath*}
		for $m_i^*=\max\{m_i,\tilde m_i\},i=1,2$, $\delta= \max\{\delta_1,\delta_2\}$ and $l=\max\{l_1+\delta l_2, \tilde l\}$.
		As $j$ tends to $+\infty,$ the intersection of all those spaces belongs to the space $C\left((0, T] ; \mathcal{S}(\R^{2n})\right) .$ This completes the proof.
	\end{proof} 
	
	\begin{lem}\label{lem5}
		Let
		\begin{linenomath*}
		$$
		\begin{aligned}
			a &\in G^{\tilde m_1,\tilde m_2}\{\tilde l_1,\tilde l_2;\delta_1\}_{int,2N} \wm + 	G^{m_1,m_2}\{l_1,l_2,l_3,l_4;\delta_2\}_{ext,N}\wm \text{ and }\\
			b &\in G^{\tilde m_1^{\prime},\tilde m_2^{\prime}}\{\tilde l_1^{\prime},\tilde l_2^{\prime};\delta_1\}_{int,2N} \wm + 	G^{m_1^{\prime},m_2^{\prime}}\{l_1^{\prime},l_2^{\prime},l_3^{\prime},l_4^{\prime};\delta_2\}_{ext,N}\wm.
		\end{aligned}
		$$
		\end{linenomath*}
		Then
		\begin{linenomath*}
		$$
		\begin{aligned}
			a b  \in & \; G^{\tilde m_1 +\tilde m_1^{\prime},\tilde m_2+\tilde m_2^{\prime}}\{\tilde l_1+\tilde l_1^{\prime},\tilde l_2+\tilde l_2^{\prime};\delta_1\}_{int,2N} \wm  \\
			&+G^{m_1+m_1^{\prime},m_2+m_2^{\prime}}\{l_1+l_1^{\prime},l_2+l_2^{\prime},l_3+l_1^{\prime},l_4+l_1^{\prime};\delta_2\}_{ext,N}\wm \\
			&+G^{\tilde m_1,\tilde m_2}\{\tilde l_1,\tilde l_2;\delta_1\}_{int,2N} \wm \cap G^{m_1^{\prime},m_2^{\prime}}\{l_1^{\prime},l_2^{\prime},l_3^{\prime},l_4^{\prime};\delta_2\}_{ext,N}\wm \\
			&+ G^{\tilde m_1^{\prime},\tilde m_2^{\prime}}\{\tilde l_1^{\prime},\tilde l_2^{\prime};\delta_1\}_{int,2N} \wm \cap G^{m_1,m_2}\{l_1,l_2,l_3,l_4;\delta_2\}_{ext,N}\wm.
		\end{aligned}
		$$
		\end{linenomath*}
	\end{lem}
		Notice that the symbols corresponding to last two summands of the above expression are non-zero only if $N< t \P\japxik<2 N,$ i.e., in $Z_{int}(2N) \cap Z_{ext}(N)$. A straightforward computation proves the above lemma.
	 
	\begin{rmk}
		If the symbols $a$ and $b$ in the previous lemma belong to
		\begin{linenomath*}
		$$
		\begin{aligned}
			G&^{\tilde m_1,\tilde m_2}\{\tilde l_1,\tilde l_2;\delta_1\}_{int,2N} \wm + 	G^{m_1,m_2}\{l_1,l_2,l_3,l_4;\delta_2\}_{ext,N}\wm \text{ and }\\
			G&^{\tilde m_1^{\prime},\tilde m_2^{\prime}}\{\tilde l_1^{\prime},\tilde l_2^{\prime};\delta_1\}_{int,2N} \wm + 	G^{m_1^{\prime},m_2^{\prime}}\{l_1^{\prime},l_2^{\prime},l_3^{\prime},l_4^{\prime};\delta_2\}_{ext,N}\wm.
		\end{aligned}
		$$
		\end{linenomath*}
		respectively, then the symbol $ab$ belongs to
		\begin{linenomath*}
		$$
		\begin{aligned}
			G&^{\tilde m_1 +\tilde m_1^{\prime},\tilde m_2+\tilde m_2^{\prime}}\{\tilde l_1+\tilde l_1^{\prime},\tilde l_2+\tilde l_2^{\prime};\delta_1\}_{int,2N} \wm  \\
			&+G^{m_1+m_1^{\prime},m_2+m_2^{\prime}}\{l_1+l_1^{\prime},l_2+l_2^{\prime},l_3+l_1^{\prime},l_4+l_1^{\prime};\delta_2\}_{ext,N}\wm,
		\end{aligned}
		$$
		\end{linenomath*}
		since there is no overlap of the regions $\pd$ and $\hyp$.
	\end{rmk}
	
	\begin{lem}
		Let $A$ and $B$ be pseudodifferential operators with symbols
		\begin{linenomath*}
		$$
		a=\sigma(A) \in G^{\tilde m_1,\tilde m_2}\{\tilde l_1,\tilde l_2;\delta_1\}_{int,2N} \wm + 	G^{m_1,m_2}\{l_1,l_2,l_3,l_4;\delta_2\}_{ext,N}\wm
		$$
		\end{linenomath*}
		and
		\begin{linenomath*}
		$$
		b=\sigma(B) \in G^{\tilde m_1^{\prime},\tilde m_2^{\prime}}\{\tilde l_1^{\prime},\tilde l_2^{\prime};\delta_1\}_{int,2N} \wm + 	G^{m_1^{\prime},m_2^{\prime}}\{l_1^{\prime},l_2^{\prime},l_3^{\prime},l_4^{\prime};\delta_2\}_{ext,N}\wm.
		$$
		\end{linenomath*}
		Then, the pseudodifferential operator $C=A \circ B$ has a symbol
		\begin{linenomath*}
		$$
		\begin{aligned}
			c=\sigma(C)  \in & \; G^{\tilde m_1 +\tilde m_1^{\prime},\tilde m_2+\tilde m_2^{\prime}}\{\tilde l_1+\tilde l_1^{\prime},\tilde l_2+\tilde l_2^{\prime};\delta_1\}_{int,2N} \wm  \\
			&+G^{m_1+m_1^{\prime},m_2+m_2^{\prime}}\{l_1+l_1^{\prime},l_2+l_2^{\prime},l_3+l_1^{\prime},l_4+l_1^{\prime};\delta_2\}_{ext,N}\wm \\
			&+G^{\tilde m_1,\tilde m_2}\{\tilde l_1,\tilde l_2;\delta_1\}_{int,2N} \wm \cap G^{m_1^{\prime},m_2^{\prime}}\{l_1^{\prime},l_2^{\prime},l_3^{\prime},l_4^{\prime};\delta_2\}_{ext,N}\wm \\
			&+ G^{\tilde m_1^{\prime},\tilde m_2^{\prime}}\{\tilde l_1^{\prime},\tilde l_2^{\prime};\delta_1\}_{int,2N} \wm \cap G^{m_1,m_2}\{l_1,l_2,l_3,l_4;\delta_2\}_{ext,N}\wm.
		\end{aligned}
		$$
		\end{linenomath*}
		and satisfies
		\begin{equation} \label{a4}
		c(t, x, \xi) \sim \sum_{\alpha \in \mathbb{N}^{n}} \frac{1}{\alpha !} \partial_{\xi}^{\alpha} a(t, x, \xi) D_{x}^{\alpha} b(t, x, \xi).
		\end{equation}
		The operator $C$ is uniquely determined modulo an operator with symbol from $C\left((0, T] ; \mathcal{S}(\R^{2n})\right).$
	\end{lem}
	
	\begin{proof}
		In view of Proposition \ref{prop1}, Proposition \ref{prop2}, Proposition \ref{prop3} and Lemma \ref{lem5} it is clear that the operator $C$ is a well-defined pseudodifferential operator. Relation (\ref{a4}) is a direct consequence of the standard composition rules (see \cite[Section 1.2]{nicRodi}).
	\end{proof}

	\begin{lem}
		Let $A$ be a pseudodifferential operator with an invertible symbol
		\begin{linenomath*}
		$$
		a=\sigma(A) \in G^{0,0}\{0,0;0\}_{int,2N}\wm + G^{0,0}\{0,0,0,0;0\}_{ext,N}\wm. 
		$$
		\end{linenomath*}
		Then, there exists a parametrix $A^{\#}$ with symbol
		\begin{linenomath*}
		$$
			a^{\#}=\sigma\left(A^{\#}\right) \in G^{0,0}\{0,0;0\}_{int,2N}\wm + G^{0,0}\{0,0,0,0;0\}_{ext,N}\wm. 
		$$
		\end{linenomath*}
	\end{lem}

	\begin{proof}
		We use the existence of the inverse of $a$ and set
		\begin{linenomath*}
		$$
		a_{0}^{\#}(t, x, \xi)=a(t, x, \xi)^{-1} \in G^{0,0}\{0,0;0\}_{int,2N}\wm + G^{0,0}\{0,0,0,0;0\}_{ext,N}\wm. 
		$$
		\end{linenomath*}
		In view of Propositions \ref{prop1} and \ref{prop2}, one can define a sequence $a_{j}^{\#}(t, x, \xi)$ recursively by
		\begin{linenomath*}
		$$
		\sum_{1 \leq|\alpha| \leq j} \frac{1}{\alpha !} \partial_{\xi}^{\alpha} a(t, x, \xi) D_{x}^{\alpha} a_{j-|\alpha|}^{\#}(t, x, \xi)=-a(t, x, \xi) a_{j}^{\#}(t, x, \xi)
		$$
		\end{linenomath*}
		with
		\begin{linenomath*}
		$$
		a_{j}^{\#} \in G^{-j,1}\{0,0;0\}_{int,2N}(\Phi^{-j},g_{\Phi,k}) + G^{-j,1}\{0,0,0,0;0\}_{ext,N} (\Phi^{-j},g_{\Phi,k}) .
		$$
		\end{linenomath*}
		Proposition \ref{prop3} then yields the existence of a symbol
		\begin{linenomath*}
		$$
		a_{R}^{\#} \in G^{0,0}\{0,0;0\}_{int,2N}\wm + G^{0,0}\{0,0,0,0;0\}_{ext,N}\wm 
		$$
		\end{linenomath*}
		and a right parametrix $A_{R}^{\#}(t, x, \xi)$ with symbol $\sigma\left(A_{R}^{\#}\right)=a_{R}^{\#} .$ We have
		\begin{linenomath*}
		$$
		A A_{R}^{\#}-I \in C\left([0, T] ; G^{-\infty,-\infty}(\omega,g_{\Phi,k})\right).
		$$
		\end{linenomath*}
		The existence of a left parametrix follows in similar lines. One can also prove the existence of a parametrix $A^{\#}$ by showing that right and left parametrix coincide up to a regularizing operator.
	\end{proof}
	
	Now, we perform a conjugation of an operator $A$ where 
	\begin{equation}\label{a5}
	  a=\sigma(A) \in G^{\tilde m_1,\tilde m_2}\{\tilde l_1,\tilde l_2;\delta_1\}_{int,2N} \wm + 	G^{m_1,m_2}\{l_1,l_2,l_3,l_4;\delta_2\}_{ext,N}\wm
	\end{equation}
	by the operator $\exp\{\Lambda(t,x,D_x) \}$  where the operator $\Lambda$ is such that
	\begin{equation}\label{a6}
		\Lambda(t,x,\xi) = \int_{0}^{t} \tilde{\psi}(r,x,\xi) dr
	\end{equation}
    for $\tilde \psi$ as defined in (\ref{tpsi}) and (\ref{bd}).
	The conjugation operation is given by 
	\begin{linenomath*}
	\[
		A_\Lambda (t,x,D_x) = e^{-\Lambda(t,x,D_x) } A(t,x,D_x) e^{\Lambda(t,x,D_x) }.
	\]
	\end{linenomath*}
	Notice that the operator $\exp\{\pm \Lambda(t,x,D_x) \}$ is a finite order pseudodifferential operator; in fact 
	\begin{equation}\label{a7}
			e^{\pm \Lambda(t,x,\xi) }  \leq (1+\P\japxik)^{\kappa_{00}},
	\end{equation}
	where $\kappa_{00}>0$ is the constant $C_{\alpha\beta}$ in (\ref{bd}) for $|\alpha|=|\beta|=0$. When $|\alpha|+|\beta|>0$, we have
	\begin{equation}\label{a8}
			 \partial_x^\beta \partial^\beta  e^{\pm \Lambda(t,x,\xi) }  \leq C^{\prime}_{\alpha\beta} e^{\pm \Lambda(t,x,\xi) } \P^{-|\beta|} \japxik^{-|\alpha|} (\log(1+\P\japxik))^{|\alpha|+|\beta|} \chi(Z_{int}(2N))
	\end{equation}
	By successive composition of the operators while performing conjugation and using Proposition \ref{prop3} and Lemma \ref{lem5}, one can prove the following lemma.
 	\begin{lem}\label{conju}
		Let the operators $A$ and $\Lambda$ be as in (\ref{a5}) and (\ref{a6}). Then
		\begin{linenomath*}
			\begin{equation}
				A_\Lambda (t,x,D_x) = A(t,x,D)+ R(t,x,D_x),
			\end{equation}
		\end{linenomath*}
		where 
		\begin{linenomath*}
		\[
			\begin{aligned}
				(\log(1+\P\japxik))^{-1} \sigma(R)(t,x,\xi)  &\in G^{\tilde m_1-1+\varepsilon,1}\{\tilde l_1,\tilde l_2;\delta_1\}_{int,2N} (\omega^{\tilde m_2}\Phi^{-1+\varepsilon}, g_{\Phi,k}) \\
				& \qquad + 	G^{m_1-1,1}\{l_1,l_2,l_3,l_4;\delta_2\}_{ext,N} (\omega^{ m_2}\Phi^{-1}, g_{\Phi,k}),
			\end{aligned}
		\]
		\end{linenomath*}
		for $\varepsilon \ll 1.$
	\end{lem}  
	
	\begin{rmk}\label{balance}
				In Lemma \ref{conju},  one can ensure a compensation for the $\varepsilon$ increase in the order of remainder symbol in the interior region by an appropriate choice of the order of singularity in the interior region. For example, the conjugation of the operator $\mathcal{D}$ in (\ref{FOS1}) yields
				\begin{linenomath*}
				\[
					\mathcal{D}_{\Lambda}(t,x,D_x) = 	\mathcal{D}(t,x,D_x) + R(t,x,D_x)
				\]
				\end{linenomath*}
				where the operator $R(t,x,D_x)$ is such that its symbol satisfies
				\begin{linenomath*}
				\[
				\begin{aligned}
					(\log(1+\P\japxik))^{-1} \sigma(R)(t,x,\xi)  &\in G^{-1+\tilde\varepsilon,1}\{0,0;0\}_{int,2N} (\omega\Phi^{-1+\tilde\varepsilon}, g_{\Phi,k}) \\
					& \qquad + 	G^{-1,1}\{0,0,1,1;\delta_1\}_{ext,N} (\omega\Phi^{-1}, g_{\Phi,k}),
				\end{aligned}
				\]
				\end{linenomath*}
				for an arbitrary small $ \tilde\varepsilon>0.$ By the definition of region $\pd,$ we have the estimate
				\begin{linenomath*}
				\[
					{(\P\japxik)^{\tilde\varepsilon}} \leq \frac{1}{t^{\tilde\varepsilon}}.
				\]
				\end{linenomath*}
				Hence, we have
				\begin{linenomath*}
				\[
				\begin{aligned}
					t^{\tilde\varepsilon}(\log(1+\P\japxik))^{-1} \sigma(R)(t,x,\xi)  &\in G^{0,1}\{0,0;0\}_{int,2N} (\omega\Phi^{-1}, g_{\Phi,k}) \\
					& \qquad + 	G^{0,1}\{0,0,1,1;\delta_1\}_{ext,N} (\omega\Phi^{-1}, g_{\Phi,k}).
				\end{aligned}
				\]
				\end{linenomath*}
	\end{rmk}

\end{document}